\documentclass[11pt, a4paper]{article}
\usepackage[textwidth = 7in, textheight=25cm,nohead]{geometry}
\usepackage{ulem}
\usepackage{arydshln}
\usepackage{graphicx, amssymb, amsmath, amsthm, amstext, booktabs, float, multirow, subfigure, rotating,natbib}
\usepackage{tikz}
\usepackage{setspace}
\usepackage{enumerate}

\numberwithin{equation}{section} 
\numberwithin{figure}{section} 
\numberwithin{table}{section} 
 
\newtheorem{lemma}{Lemma}[section] 

\newtheorem{theorem}{Theorem}[section] 
\newtheorem{assumption}{Assumption}[section] 
\newtheorem{proposition}{Proposition}[section] 
\newtheorem{corollary}{Corollary}[section]

\theoremstyle{remark} 
\newtheorem{remark}{Remark}[section] 

\theoremstyle{definition} 
\newtheorem{example}{Example}[section]

\bibpunct{(}{)}{,}{a}{,}{,}

\def\ppn{\vskip 6pt \noindent }
\def\R{{\mathbb{R}}}
\def\N{{\mathbb{N}}}
\def\P{{\mathbb{P}}}
\def\E{{\mathbb{E}}}
\newcommand{{\Xs}}{{\cal X}}
\newcommand{{\Ys}}{{\cal Y}}
\newcommand{{\Ls}}{{\cal L}}
\newcommand{{\Ss}}{{\cal S}}
\newcommand{{\Ms}}{{\cal M}}
\newcommand{{\Gs}}{{\cal G}}
\newcommand{{\Hs}}{{\cal H}}
\newcommand{{\Ns}}{{\cal N}}
\newcommand{{\Is}}{{\cal I}}
\newcommand{{\Bs}}{{\cal B}}
\newcommand{{\Cs}}{{\cal C}}
\newcommand{{\Rs}}{{\cal R}}
\newcommand{{\Us}}{{\cal U}}
\newcommand{{\pp}}{{\mathbf p}}
\newcommand{{\KK}}{{\mathbf K}}
\newcommand{{\HH}}{{\mathbf H}}
\newcommand{{\II}}{{\mathbf I}}
\newcommand{{\yy}}{{\mathbf y}}
\newcommand{{\ab}}{{\mathbf a}}
\newcommand{\C}{\mathbb{C}}

\newcommand\encircle[1]{%
  \tikz[baseline=(X.base)] 
    \node (X) [draw, shape=circle, inner sep=0] { #1};}

\newcommand{{\toL}}{{\overset{\mathcal{L}}{\longrightarrow}\ }}

\newcommand{{\MC}}{{\,  *_{\text{\scalebox{0.65}{$\Ms$}}}\,  }}

\newcommand{{\dou}}{$\leadsto$\ }

\DeclareMathOperator{\var}{\mathbb{V}ar}

\DeclareMathOperator{\bias}{\mathbb{B}ias}

\newcommand{\indic}[1]{
\hbox{${\it 1}\hskip -4.5pt I_{\{ #1 \}}$}
}

\begin{document}

\setlength{\belowdisplayskip}{5pt} \setlength{\belowdisplayshortskip}{3pt}
\setlength{\abovedisplayskip}{5pt} \setlength{\abovedisplayshortskip}{0pt}

\title{Mellin-Meijer-kernel density estimation on $\R^+$}
\author{\sc{Gery Geenens}\thanks{email: {\tt ggeenens@unsw.edu.au}, tel +61 2 938 57032, fax +61 2 9385 7123 }\\School of Mathematics and Statistics,\\ UNSW Sydney, Australia }
\date{\today}
\maketitle
\thispagestyle{empty} 

%\begin{center} {\bf Preliminary and Incomplete draft: do not cite or circulate without the author's permission.} \end{center}

\begin{abstract}
\noindent Nonparametric kernel density estimation is a very natural procedure which simply makes use of the smoothing power of the convolution operation. Yet, it performs poorly when the density of a positive variable is to be estimated (boundary issues, spurious bumps in the tail). So various extensions of the basic kernel estimator allegedly suitable for $\R^+$-supported densities, such as those using Gamma or other asymmetric kernels, abound in the literature. Those, however, are not based on any valid smoothing operation analogous to the convolution, which typically leads to inconsistencies. By contrast, in this paper a kernel estimator for $\R^+$-supported densities is defined by making use of the Mellin convolution, the natural analogue of the usual convolution on $\R^+$. From there, a very transparent theory flows and leads to new type of asymmetric kernels strongly related to Meijer's $G$-functions. The numerous pleasant properties of this `Mellin-Meijer-kernel density estimator' are demonstrated in the paper. Its pointwise and $L_2$-consistency (with optimal rate of convergence) is established for a large class of densities, including densities unbounded at 0 and showing power-law decay in their right tail. Its practical behaviour is investigated further through simulations and some real data analyses.
\end{abstract}

\section{Introduction}\label{sec:intro}

Kernel density estimation is a very popular nonparametric method which enables estimation of an unknown probability density function without making any assumption on its functional shape. Its main ingredients are a {\it kernel} function $K$, typically a unit-variance probability density symmetric around 0, and a {\it smoothing parameter} $h>0$ fixed by the analyst, which controls the smoothness of the resulting estimate. One usually defines $K_h(\cdot) = K(\cdot/h)/h$, the rescaled version of $K$ which has standard deviation $h$. A common choice for $K$ is the standard normal density $\phi$. Then, the estimator simply makes use of the well-known smoothing power of the convolution operation. Specifically, consider a sample $\{X_k,k=1,\ldots,n\}$ drawn from a distribution $F$ admitting a density $f$, and define its empirical measure $\P_n = \frac{1}{n}\sum_{k=1}^n \delta_{X_k}$, where $\delta_x(\cdot) = \delta( \cdot -x)$ and $\delta$ is the usual Dirac delta. The conventional kernel density estimator of $f$ is just
\begin{equation} \hat{f}(x) = \left(K_h * \P_n \right)(x). \label{eqn:convolkde} \end{equation}
Expanding this convolution yields the familiar expression:
\begin{equation} \hat{f}(x) = \int_{-\infty}^{+\infty} K_h(x-u) \frac{1}{n}\sum_{k=1}^n \delta_{X_k}(u)\,du = \frac{1}{n} \sum_{k=1}^n K_h(x-X_k) = \frac{1}{nh} \sum_{k=1}^n K\left(\frac{x-X_k}{h}\right). \label{eqn:convkde} \end{equation}
The statistical properties of $\hat{f}$ are well understood \citep{Wand95,Hardle04}, and its merit is widely recognised when the support of $F$ is the whole real line $\R$. Unfortunately, when the support of $F$ admits boundaries, the good properties of (\ref{eqn:convkde}) are usually lost.

\ppn A case of bounded support of major importance is when $F$ is the distribution of a positive random variable $X$, with density $f$ supported on $\R^+ = [0,+\infty)$. Typically, those distributions are skewed, with $f$ showing a maximum at or near the boundary 0 and a long tail on the right side. Sometimes, the behaviour of the density close to 0 is what mostly matters for the analyst; in other cases, it is rather the tail behaviour which is the main focus, for instance when high quantiles (e.g., Value-at-Risk) are of interest. Yet, (\ref{eqn:convkde}) fails to correctly estimate both the behaviour of $f$ close to 0 and in the tail. Close to 0, the estimator suffers from {\it boundary bias}: the terms $K_h(x-X_k)$ corresponding to $X_k$'s close to 0 typically overflow beyond the boundary and place positive probability mass in the forbidden area, generally preventing consistency of the estimator there \citep[Section 2.11]{Wand95}. In the tail region, where data are usually sparse, it produces `{\it spurious bumps}' \citep{Hall04}, i.e.\ artificial local maxima at each observed value, thus performing poorly as well. %These two problems explain, at least partially, why kernel density estimation has not been used more to date in applied research, where positive variables are numerous. 

\ppn Hence modifications and extensions of (\ref{eqn:convkde}), attempting to make it suitable for  $\R^+$-supported densities, abound in the literature. Early attempts at curing boundary effects looked for correcting $\hat{f}$ close to 0. Those include the `cut-and-normalised' method and its variants based on `boundary kernels' \citep{Gasser79,Muller91,Jones93,Jones96,Cheng97,Zhang99,Dai10}, the reflection method \citep{Schuster85,Karunamuni05}, and other types of local data alteration \citep{Cowling96,Hall02,Park03}. These methods are essentially {\it ad hoc} manual surgeries on (\ref{eqn:convkde}) close to 0, and have shown their limitations. In addition, as they leave the tail area untouched, they do not address the `spurious bumps' at all.

\ppn Later, the problem was approached from a more global perspective and (\ref{eqn:convkde}) was generalised as
\begin{equation} \hat{f}(x) = \frac{1}{n} \sum_{k=1}^n L_h(X_k;x), \label{eqn:asymkde} \end{equation}
where $L_h(\cdot;x)$ is an asymmetric $\R^+$-supported density whose parameters are functions of $x$ and a smoothing parameter $h>0$. Using asymmetric kernels supposedly enables the estimator to take the constrained nature of the support of $f$ into account. In his pioneering work, \cite{Chen00} took $L_h(\cdot;x)$ to be the Gamma density with shape parameter $\alpha = 1+x/h^2$ and rate $\beta = 1/h^2$, defining the `first' Gamma kernel density estimator, viz. 
\begin{equation} \hat{f}(x) =   \frac{1}{n} \sum_{k=1}^n \frac{X_k^{x/h^2}e^{-X_k/h^2}}{h^{2x/h^2+2}\Gamma\left(x/h^2+1\right)} \label{eqn:gammakde} \end{equation}
(we use $h^2$ instead of Chen's original $b$ for the smoothing parameter, for consistency with standard notation). Although more types of asymmetric kernels were investigated in the subsequent literature (Log-Normal, \citet{Jin03,Igarashi16}; Birnbaum-Saunders, \citet{Jin03,Marchant13,Igarashi14}; Inverse Gaussian and reciprocal Inverse Gaussian distributions, \citet{Scaillet04,Igarashi14}), none really outperformed Chen's Gamma kernel density estimator which remains some sort of `gold standard' in the field. Its properties were further investigated in \cite{Bouezmarni05,Hagmann07,Zhang10} and \cite{Malec12}. Asymmetric kernel density estimation remains an area of very active research, as the number of recent papers in the area evidences \citep{Kuruwita10,Jeon13,Dobrovidov14,Igarashi14,Hirukawa14,Funke15,Markovic15,Hoffman15,Funke16,Igarashi16,Markovic16,Rosa16,Balakrishna17}. \cite{Hirukawa15} describe a family of `generalised Gamma kernels' which includes a variety of similar asymmetric kernels in an attempt to standardise those results. 

\ppn `Ironically', as \citet[Section 2]{Jones07} put it, such asymmetric kernel estimators {\it do not} really address boundary problems. Indeed, generally nothing prevents (\ref{eqn:asymkde}) from taking positive values for $x<0$. For instance, (\ref{eqn:gammakde}) is defined and positive for $x<0$ as long as $1+x/h^2 >0$. This explains why those estimators need a further correction near the boundary (`second' Gamma kernel estimator in \cite{Chen00}, also known as `modified' Gamma kernel estimator; see also Conditions 1 and 2 in \cite{Hirukawa15}), performing yet another `manual surgery' on an initially unsuitable estimator around 0. Worse, even the modified version of the Gamma kernel estimator was somewhat picked apart in \cite{Zhang10} and \cite{Malec12}.
 
\ppn Actually, those problems arise from the fact that estimators like (\ref{eqn:asymkde}) are, in general, not induced by any valid smoothing operation analogous to (\ref{eqn:convolkde}) on $\R^+$. Beyond questioning the mere validity of the construction, this has some unpleasant consequences. Those include that (\ref{eqn:asymkde}) does not automatically integrate to one, hence is not a {\it bona fide} density (and manually rescaling the estimate usually produces some extra bias). Also, in contrast to the obvious $(x-X_k)/h$ visible in (\ref{eqn:convkde}), it is not clear how (\ref{eqn:asymkde}) actually appreciates the proximity between $x$ and the observations $X_k$'s relative to $h$. The local nature of (\ref{eqn:asymkde}) is only induced by making the parameters of the kernel $L_h(\cdot;x)$ heuristically depend on $x$ and $h$ in a way barely driven by any intuition: see for instance (\ref{eqn:gammakde}) or \citet[Section 2.3.1]{Hirukawa15} for more general expressions. This spoils the intuitive simplicity of the initial kernel density estimation scheme.

\ppn These observations demonstrate the need for an asymmetric kernel density estimator based on a simple, natural and  transparent methodology, and with good theoretical {\it and} practical properties. This paper precisely suggests and studies such a methodology, finding its inspiration from yet another popular approach for kernel estimation of $\R^+$-supported densities: support transformation \citep{Copas80,Silverman86,Wand91,Marron94,Ruppert94,Geenens14,Geenens16}. Define $Y = \log(X)$. Kernel estimation of the density $g$ of $Y$ should be free from boundary issues, as $Y$ is supported on $\R$. From standard arguments, one has $f(x) = \frac{g(\log x)}{x}, x>0$, which suggests, upon estimation of $g$ by some estimator $\hat{g}$, the estimator $\hat{f}(x) = \frac{\hat{g}(\log x)}{x}$ for $f$. If one uses a basic kernel estimator like (\ref{eqn:convkde}) for estimating $g$, one obtains the closed form
\begin{equation} \hat{f}(x) = \frac{1}{nhx} \sum_{k=1}^n K\left(\frac{\log x-\log X_k}{h}\right). \label{eqn:logkde} \end{equation}
With the Gaussian kernel $K = \phi$ in (\ref{eqn:logkde}), one gets
\begin{equation} \hat{f}(x) = \frac{1}{nhx} \sum_{k=1}^n \phi\left(\frac{\log x-\log X_k}{h}\right) = \frac{1}{n} \sum_{k=1}^n \frac{1}{x\sqrt{2\pi h^2}}\exp\left(-\frac{(\log x-\log X_k)^2}{2h^2}\right). \label{eqn:lognormkde} \end{equation}
Interestingly, this can be written 
\begin{equation} \hat{f}(x) = \frac{1}{n} \sum_{k=1}^n L_h(x;X_k) \label{eqn:asymkde2} \end{equation}
where $L_h(\cdot;X_k)$ is the log-Normal density with parameters $\mu = \log X_k$ and $\sigma = h$. It seems, therefore, fair to call this estimator (\ref{eqn:lognormkde}) the {\it log-Normal kernel density estimator} -- remarkably, it is different to \cite{Jin03}'s and \cite{Igarashi16}'s homonymous estimators. It can be shown \citep{Geenens16} that, under suitable conditions,
\begin{align}
 \E\left(\hat{f}(x) \right) & = f(x) + \frac{1}{2} h^2 (x^2 f''(x) + 3x f'(x) + f(x)) + o(h^2) \label{eqn:asymptbiaslognormkde} \\
\var\left(\hat{f}(x)\right) & = \frac{f(x)}{ 2\sqrt{\pi}\,xnh} +o((nh)^{-1}),  \label{eqn:asymptvarlognormkde}
\end{align}
as $n \to \infty$, $h \to 0$ and $nh \to \infty$. However, and despite being simple and natural, (\ref{eqn:logkde})-(\ref{eqn:lognormkde}) shows very disappointing performance in practice, see for instance Figure 2.13 in \cite{Silverman86} or Figure 2.1 in \cite{Geenens16}. This estimator as-is has consequently been given little support in the literature.

\ppn There are, however, two important observations to make. The first is that this `log-Normal kernel density estimator' as (\ref{eqn:asymkde2}) is an asymmetric kernel density estimator but of a different nature to (\ref{eqn:asymkde}). Here $L_h(\cdot;X_k)$ is an asymmetric $\R^+$-supported density whose parameters are functions of $X_k$ and a smoothing parameter $h>0$, that is, the roles of $x$ and $X_k$ have been swapped around compared to (\ref{eqn:asymkde}). Surprisingly, estimators of type (\ref{eqn:asymkde2}) have not been investigated much in the literature, two notable exceptions being \cite{Jeon13} and the discusssion in \cite{Hoffman15}. Yet, (\ref{eqn:asymkde2}) is as valid a generalisation of (\ref{eqn:convkde}) as (\ref{eqn:asymkde}): $K$ being symmetric in (\ref{eqn:convkde}), $x$ and $X_k$ can be switched imperceptibly. For instance, in the case $K=\phi$, it is actually irrelevant whether $h^{-1}\phi((x-X_k)/h)$ is $\phi_{X_k,h}(x)$ or $\phi_{x,h}(X_k)$ (where $\phi_{\mu,\sigma}$ is the $\Ns(\mu,\sigma^2)$-density). In the asymmetric case, though, the respective roles of $x$ and $X_k$ are of import. 

\ppn Working with $L_h$ a proper density in $x$ is more natural, though. Clearly, $\int \hat{f}(x)\,dx = 1$ automatically, and $L_h(x;X_k) \equiv 0$ for $x<0$. Hence (\ref{eqn:asymkde2}) {\it cannot} assign probability weight to the negative values, which attacks the above-described boundary problem at the source. Also, $L_h(x;X_k)$ typically shares the same right-skewness as the density to be estimated, and this is usually beneficial to the estimator in the tail area (no more `spurious bumps'). It is not clear why estimators of type (\ref{eqn:asymkde2}) have remained so inconspicuous in the literature so far. 

\ppn The second important observation is how (\ref{eqn:logkde}) really addresses the boundary isssue. This will be discussed in detail in Section \ref{subsec:LNkde}, and will suggest a simple redefinition of (\ref{eqn:convolkde}) suitable for $\R^+$-supported densities. From there, a very natural methodology will flow and will lead to a new asymmetric kernel density estimator of type (\ref{eqn:asymkde2}), based on a valid and intuitive smoothing operation on $\R^+$. As it will be seen, the new estimator shares some similarities with the `convolution power kernel density estimator' of \cite{Comte12} and with the `varying kernel density estimator' of \cite{Mnatsakanov12}. The latter estimator actually arose as a by-product from those authors' previous work on recovering a probability distribution from its moments \citep{Mnatsakanov03,Mnatsakanov08b}. The appropriate tool for solving that so-called {\it Stieltjes moment problem} was seen to be the {\it Mellin transform} \citep{Mellin96}, and our estimator also has a strong `Mellin flavour'. The kernel functions that fit naturally in this framework belong to a family of $\R^+$-supported distributions strongly related to Meijer's $G$-functions \citep{Meijer36,Erdelyi53} -- the duality between the Mellin transform and Meijer's $G$-functions was elucidated in \cite{Marichev82}. Hence we call the whole methodology {\it Mellin-Meijer-kernel density estimation} on $\R^+$. The numerous pleasant properties of the estimator will be exhibited throughout the paper: it relies on a natural smoothing operation on $\R^+$, whereby it avoids any inconsistency; it has a closed form expression easy-to-understand intuitively; it always produce a {\it bona fide} density; the smoothness of its estimates is controlled by a natural smoothing parameter; it is consistent for densities potentially unbounded at the boundary $x=0$ and admitting power-law decay in the right tail; it reaches optimal MISE-rate of convergence over $\R^+$ under mild assumptions; etc. In addition, its theoretical properties are derived directly through the Mellin transform theory, which make all the proofs very transparent. Properties of the estimator `in the Mellin world' also suggest an easy way of selecting the always crucial smoothing parameter in practice. 

\ppn The paper is structured as follows: Section \ref{sec:motiv} motivates the `Mellin-Meijer' construction and revise the main properties of the Mellin transform and the Meijer $G$-functions. Section \ref{sec:Mkde} defines the estimator, while in Section \ref{subsec:asympt} its theoretical properties are derived. Section \ref{sec:smoothpar} suggests an easy way to select the smoothing parameter. Section \ref{sec:sim} and Section \ref{sec:realdat} investigate the performance of the estimator in practice, through simulations and real data examples, respectively. Section \ref{sec:ccl} summarises the main ideas of the paper and offers some perpectives for continuing this line of research. Futher properties, proofs and technical lemmas are provided in Appendix.

\section{Preliminaries} \label{sec:motiv}

\subsection{Motivation} \label{subsec:LNkde}

The very origin of the boundary issues of (\ref{eqn:convkde}) is easy to understand. The convolution of two probability densities $g_1 * g_2$ is known to be the density of the sum of two independent random variables having respective densities $g_1$ and $g_2$. Hence, smoothing is achieved in (\ref{eqn:convolkde}) through `diluting' each observation $X_k$ by {\it adding} to it some continuous random noise $\varepsilon$ with density $K_h$. This does not cause any inconsistency if $f$ is $\R$-supported, but it does if $f$ is $\R^+$-supported. Indeed, $X_k + \varepsilon$ may take on negative values, and will surely do so for the small $X_k$'s. This produces an estimated density $\hat{f}$ that `spills over'. In response to that, (\ref{eqn:logkde}) first sends the observations onto the whole $\R$ through the log-transformation, adds to them some random disturbance of mean 0 and variance $h^2$, and moves everything back to $\R^+$ by exponentiation. So, in terms of what happens on $\R^+$, (\ref{eqn:logkde}) realises smoothing by {\it multiplying} each observation $X_k$ by a positive random disturbance $\varepsilon$. In the particular case of (\ref{eqn:lognormkde}), $\varepsilon$ has a certain log-Normal distribution. 

\ppn In algebraic terms, the conventional estimator (\ref{eqn:convolkde})-(\ref{eqn:convkde}) is justified on $\R$ because $(\R,+)$ is a group. By contrast $(\R^+,+)$ is not, which causes issues for $\R^+$-supported densities. Now, $(\R^+,\times)$ {\it is} a group. Naturally, it is isomorphic to $(\R,+)$ through the $\log$ transformation, and this is what theoretically validates estimator (\ref{eqn:logkde}). However, the benefits of distorting $(\R^+,\times)$ to forcibly move back to $(\R,+)$ are not clear -- the presence of $f$ and $f'$ in the bias expression (\ref{eqn:asymptbiaslognormkde}) is precisely caused by that distortion (the bias of the conventional estimator (\ref{eqn:convkde}) only involves $f''$). It seems more natural to define a kernel estimator directly in $(\R^+,\times)$, which would address in a forthright manner the particular challenges arising in that environment.  

\ppn The estimator suggested in this paper will thus realise smoothing by {\it multiplying} each observation $X_k$ by a random disturbance $\varepsilon$ whose density is supported on $\R^+$, generalising by a large extent the log-Normal kernel density estimator (\ref{eqn:lognormkde}). It will heavily rely on properties of the Mellin transform, the ``{\it natural analytical tool to use in studying the distribution of products and quotients of independent random variables}'',  as \cite{Epstein48} described it. The next section briefly reviews the properties of that transform which will be useful in this paper.

\subsection{Mellin transform and Mellin convolution} \label{subsec:Mellin}

The {\it Mellin transform} of any locally integrable $\R^+$-supported function $f$ is the function defined on the complex plane $\C = \{z:z=c+i \omega; c,\omega \in \R\}$ as 
\begin{equation} \Ms(f;z) = \int_0^\infty x^{z-1} f(x) \,dx, \label{eqn:Mellindef} \end{equation} 
when the integral converges. The change of variable $x = e^{-u}$ shows that the Mellin transform is directly related to the Laplace and Fourier transforms and suggests that it is, in some sense, equivalent to them. This is only partly true, as \cite{Butzer14} noted. There are many situations where it appears more convenient to stick to the Mellin form, and the problem studied here is surely one of those. If, for some $\delta > 0$ and $a < b$, 
\begin{equation} f(x)= O(x^{-(a-\delta)}) \text{ as } x \to 0^+ \text{ and } f(x)=O(x^{-(b+\delta)}) \text{ as } x \to +\infty, \label{eqn:Sfdef} \end{equation}
then (\ref{eqn:Mellindef}) converges absolutely on the vertical strip of the complex plane $ \Ss_f =\{z \in \C: a < \Re(z) < b\}$.  It can be shown that $\Ms(f;\cdot)$ is holomorphic on $\Ss_f$ -- therefore known as the {\it strip of holomorphy} of $\Ms(f;\cdot)$ -- and uniformly bounded on any closed vertical strip of finite width entirely contained in $\Ss_f$. There is a one-to-one correspondence between the function $f$ and {\it the couple} $(\Ms(f;\cdot),\Ss_f)$, in the sense that two different functions may have the same Mellin transform, but defined on two non-overlapping vertical strips of the complex plane. For instance, the function $e^{-x}$ has, by definition, Euler's Gamma function $\Gamma(z)$ as Mellin transform for $\Re(z)>0$, but for any real nonnegative integer $n$, the restriction of $\Gamma(z)$ to $\{z \in \C: -(n+1) < \Re(z) < -n\}$ is actually the Mellin transform of 
\begin{equation} f_n(x) \doteq e^{-x} + \sum_{k=0}^n (-1)^{k+1} \frac{x^k}{k!}, \label{eqn:CSGam} \end{equation}
by the {\it Cauchy-Saalsch\"utz representation} of $\Gamma(z)$ \citep[Section 3.2.2]{Temme96}. It is thus equivalent to know $f$ or $\Ms(f;\cdot)$ {\it in a given vertical strip of $\C$}. In particular, $f$ can be recovered from $\Ms(f;\cdot)$ by the {\it inverse Mellin transform}: for all $ x>0$,
\begin{equation} f(x) = \frac{1}{2\pi i} \int_{\Re(z)=c} x^{-z} \Ms(f;z) \,dz, \label{eqn:invMellin} \end{equation}
for any real $c \in \Ss_f$. From Cauchy's residue theorem, the integration path can be displaced sideways inside $\Ss_f$ without affecting the result of integration, so the value of (\ref{eqn:invMellin}) is independent of the particular constant $c \in \Ss_f$, and the integral is absolutely convergent. See \citet[Chapter 4]{Sneddon74}, \citet[Chapter 3]{Wong01}, \citet[Chapter 3]{Paris01}, \citet[Chapter 6]{Graf10} or \citet[Section VIII.13]{Godement15} for comprehensive treatments of the Mellin transform. A exhaustive table of Mellin transforms and inverse Mellin transforms is provided in \citet[Chapters VI and VII]{Bateman54}.

\ppn Now, if $f$ is the probability density of a positive random variable $X$, then 
\begin{equation} \Ms(f;1) = \int_0^\infty f(x)\,dx = 1.  \label{eqn:Mell1} \end{equation}
Hence the line $\{z \in \C:\Re(z) = 1\}$ always belongs to the strip of holomorphy of Mellin transforms of all probability density functions supported on $\R^+$. So there will never be any ambiguity about the strip of holomorphy here, which will allow $f$ to be unequivocally represented by its Mellin transform $\Ms(f;\cdot)$, and vice-versa. In the above example, $f(x)=e^{-x} \indic{x>0}$ is the only probability density with Mellin transform $\Gamma(z)$, as that is the only $f$ such that $\Ms(f;z) =\Gamma(z)$ {\it and} $\Ss_f \ni 1$. This also makes clear the essential role of the value $z=1$ in this framework. 

\ppn From (\ref{eqn:Mellindef}), one clearly has
\begin{equation} \Ms(f;z) = \E(X^{z-1}), \label{eqn:Mellmoments} \end{equation}
thus $\Ms(f;\cdot)$ actually defines all real, complex, integral and fractional moments of $X$. Hence, for $f$ a probability density, the strip of holomorphy of $\Ms(f;\cdot)$ is actually determined by the existence (finiteness) of the real moments of $f$: 
\begin{equation} z \in \Ss_f  \iff \E(X^{\Re(z) - 1}) < \infty. \label{eqn:Sfmom} \end{equation}
In view of (\ref{eqn:Sfdef}), $b = \infty$ for light-tailed densities whose all positive moments exist, while $\Ss_f$ is bounded from the right ($1 < b < \infty$) for fat-tailed densities with only a certain number of finite positive moments.\footnote{The qualifiers `fat', `heavy' or `long'  have sometimes found different meanings in the literature when describing the tails of a distribution. In this paper, by `fat-tailed' distribution we mean explicitly this: a distribution whose not all positive power moments are finite. Hence here we consider the log-Normal as `light-tailed', although it is is generally regarded as `heavy-tailed' in many other references.} Similarly, $a = -\infty$ for densities whose all negative moments exist -- let us call such densities `{\it light-headed}', while $\Ss_f$ is bounded from the left ($-\infty<a < 1$) for `{\it fat-headed}' densities, for which some negative moments are infinite. 

\ppn From (\ref{eqn:Mellmoments}), and given that moments of products of independent random variables are the products of the individual moments, the density $g$ of the product of two independent positive random variables with respective densities $g_1$ and $g_2$, has Mellin transform 
\begin{equation} \Ms(g;z) = \Ms(g_1;z) \Ms(g_2;z). \label{eqn:MellProd} \end{equation}
From standard arguments, one also knows that
\begin{equation} g(x) = \int_{0}^\infty g_1\left(\frac{x}{v}\right)g_2(v)\,\frac{dv}{v}. \label{eqn:Mellinconv}
\end{equation}
This operation, that we will denote $g(x) = (g_1 \MC g_2)(x)$, is called {\it Mellin convolution}, owing to the equivalence
\begin{equation} g(x) = (g_1 \MC g_2)(x) \quad \iff \quad \Ms(g;z) = \Ms(g_1;z) \Ms(g_2;z) \quad \text{ for } z \in \Ss_{g_1} \cap \Ss_{g_2}. \label{eqn:prodMell} \end{equation}
Clearly, Mellin transform/Mellin convolution play the same role for products of independent random variables as Fourier transform/convolution for sums of independent random variables. %From this interpretation we see that $\MC$ is commutative and associative: $g_1 \MC g_2 = g_2 \MC g_1$ and $g_1 \MC (g_2 \MC g_3) = (g_1 \MC g_2) \MC g_3$. 

\ppn Some important operational properties of the Mellin transform can be found in Appendix \ref{app:A}. Those are useful for easily obtaining the Mellin transform of common $\R^+$-supported probability densities.

\begin{example} \label{ex:Gamma} {\bf Gamma density:} Consider the Gamma$(\alpha,\beta)$-density, i.e.,
\begin{equation} f_\text{G}(x) = \frac{\beta^\alpha}{\Gamma(\alpha)}x^{\alpha-1} e^{-\beta x}, \qquad x > 0. \label{eqn:gamdens} \end{equation}
By definition, $\Ms(e^{-x};z) = \Gamma(z)$ ($\Re(z) > 0$). By (\ref{eqn:A2}), it follows $\Ms(e^{-\beta x};z)= \beta^{-z} \Gamma(z)$ ($\Re(z) > 0$), and by (\ref{eqn:A4}), $\Ms(x^{\alpha-1} e^{-\beta x};z) = \beta^{-(z+\alpha-1)}\Gamma(z+\alpha-1)$ ($\Re(z) > 1-\alpha$). Finally, by (\ref{eqn:A1}), we obtain
\begin{equation} \Ms(f_\text{G};z) = \frac{\beta^\alpha}{\Gamma(\alpha)}\beta^{-(z+\alpha-1)}\Gamma(z+\alpha-1) = \frac{1}{\beta^{z-1}}\frac{\Gamma(\alpha + z-1)}{\Gamma(\alpha)}, \qquad  \Re(z) > 1-\alpha. \label{eqn:gamMT} \end{equation}
The strip of holomorphy is bounded from the left, but not from the right: the Gamma density is fat-headed and light-tailed, it has all its positive moments but only those negative moments $\E(X^{-r})$ for $r < \alpha$.
\end{example}

\begin{example}{\bf Inverse Gamma density:} Consider the Inverse Gamma$(\alpha,\beta)$-density $f_\text{IG}$, i.e.\ the density of the random variable $Y = 1/X$ if $X$ has the above Gamma$(\alpha,\beta)$-distribution. Standard results on functions of random variables show that \begin{equation} f_\text{IG}(x) = \frac{1}{x^2}f_\text{G}\left(\frac{1}{x}\right), \label{eqn:fIGfG}\end{equation}
that is,
\begin{equation} f_\text{IG}(x) = \frac{\beta^\alpha}{\Gamma(\alpha)}x^{-\alpha-1} e^{-\beta/x}, \qquad x > 0. \label{eqn:invgamdens} \end{equation}
Combining (\ref{eqn:A3}) and (\ref{eqn:A4}), it follows directly from (\ref{eqn:fIGfG}) that $\Ms(f_\text{IG};z) = \Ms(f_\text{G};2-z)$. The Mellin transform of (\ref{eqn:invgamdens}) is thus
\begin{equation} \Ms(f_\text{IG};z) = \frac{1}{\beta^{1-z}}\frac{\Gamma(\alpha + 1-z)}{\Gamma(\alpha)}, \qquad \Re(z) < 1+ \alpha. \label{eqn:invgamMT} \end{equation}
Its strip of holomorphy is bounded from the right, but not from the left. The Inverse Gamma distribution is light-headed and fat-tailed, as expected from its definition and Example \ref{ex:Gamma}.
\end{example}

\ppn The Gamma and Inverse Gamma densities are just two very particular cases of a huge class of $\R^+$-supported densities which have Mellin transforms of similar tractable form, viz.\ (rescaled) ratios of Gamma functions. Given that the whole kernel density estimation methodology proposed in this paper will rely on Mellin-convolution ideas, those densities are the natural candidates for acting as `kernel' in our framework. We define such densities, called {\it Meijer densities}, in Section \ref{subsec:Meijerdens}, after a brief review of Meijer's $G$-functions.

\subsection{Meijer's $G$-functions} \label{sec:MeijerGkern}

\cite{Meijer36} introduced the $G$-functions as generalisations of the hypergeometric functions. There are three types of $G$-functions, one of them being the following {\it Barnes integral}: for $x >0$,
\[G_{p,q}^{m,n}\left(x \left| \substack{a_1,\ldots,a_p \\ b_1,\ldots,b_q }\right.  \right) = \frac{1}{2\pi i} \int_{\Re(z)=c'} \frac{\prod_{j=1}^m \Gamma(b_j -z) \prod_{j=1}^n \Gamma(1 - a_j +z)}{\prod_{j=m+1}^q \Gamma(1-b_j + z) \prod_{j=n+1}^p \Gamma(a_j -z)} x^{z}  \,dz, \] 
for $0 \leq m \leq p$ and $0 \leq n \leq q$ ($m,n,p,q$ are natural numbers), and for some constant $c'$ ensuring the existence of the integral. Through the change of variable $z \to -z$, this is also
\begin{equation} G_{p,q}^{m,n}\left(x \left| \substack{a_1,\ldots,a_p \\ b_1,\ldots,b_q }\right.  \right) = \frac{1}{2\pi i} \int_{\Re(z)=c} \frac{\prod_{j=1}^m \Gamma(b_j +z) \prod_{j=1}^n \Gamma(1 - a_j -z)}{\prod_{j=m+1}^q \Gamma(1-b_j - z) \prod_{j=n+1}^p \Gamma(a_j +z)} x^{-z}  \,dz, \label{eqn:Gfunc} \end{equation} 
with $c=-c'$, which shows by (\ref{eqn:invMellin}) that $G_{p,q}^{m,n}\left(x \left| \substack{a_1,\ldots,a_p \\ b_1,\ldots,b_q }\right.  \right)$ has Mellin transform
\begin{equation} \frac{\prod_{j=1}^m \Gamma(b_j +z) \prod_{j=1}^n \Gamma(1 - a_j -z)}{\prod_{j=m+1}^q \Gamma(1-b_j - z) \prod_{j=n+1}^p \Gamma(a_j +z)} \label{eqn:fullMeijer} \end{equation}
on some strip of the complex plane containing $\{z \in \C: \Re(z) =c\}$. In particular, the $G_{1,1}^{1,1}\left(x \left| \substack{a_1 \\ b_1}\right.  \right)$-function has Mellin transform 
\begin{equation} 
 \Ms\left(G_{1,1}^{1,1}\left(\cdot \left| \substack{a_1 \\ b_1}\right.  \right);z\right) = \Gamma(b_1+z)\Gamma(1-a_1-z) \label{eqn:MTG11}
\end{equation}
on the strip $\{z \in \C: -b_1<\Re(z)<1-a_1 \}$ (provided $a_1-b_1<1$). The $G$-functions are very general functions whose particular cases cover most of the common, useful or special functions defined on $\R^+$. See \citet[Section 5.3]{Erdelyi53} or \cite{Mathai73} for more details, or \cite{Beals13} for a recent short review.

\subsection{Meijer densities} \label{subsec:Meijerdens}

For some $\nu,\gamma,\xi > 0$ and $\theta \in (0,\pi/2)$, consider the $\R^+$-supported function $L_{\nu,\gamma,\xi,\theta}$ whose Mellin transform is
\begin{equation} \Ms(L_{\nu,\gamma,\xi,\theta};z) = \nu^{z-1} \left(\frac{1}{\tan^2 \theta} \right)^{\xi(z-1)}  \frac{\Gamma\left(\frac{\xi^2}{\gamma^2 \cos^2 \theta}+\xi(z-1) \right)\Gamma\left(\frac{\xi^2}{\gamma^2 \sin^2 \theta}+\xi(1-z) \right)}{\Gamma\left(\frac{\xi^2}{\gamma^2 \cos^2 \theta}\right)\Gamma\left(\frac{\xi^2}{\gamma^2 \sin^2 \theta}\right)}\label{eqn:MTkern} \end{equation}
on the strip of holomorphy 
\begin{equation} \Ss_{L_{\nu,\gamma,\xi,\theta}} = \left\{z \in \C: 1-\frac{\xi}{\gamma^2 \cos^2 \theta} < \Re(z) < 1+\frac{\xi}{\gamma^2 \sin^2 \theta}\right\}. \label{eqn:SL}  \end{equation} 
Clearly, for all $\nu,\gamma,\xi > 0$ and $\theta \in (0,\pi/2)$, $\Ms(L_{\nu,\gamma,\xi,\theta};1) =1$, hence $\int_0^\infty L_{\nu,\gamma,\xi,\theta}(x)\,dx = 1$, by (\ref{eqn:Mell1}). In fact, the following result shows that $L_{\nu,\gamma,\xi,\theta}$ is always a valid probability density.

\begin{proposition} \label{prop:Meijerkern} For all $\nu,\gamma,\xi > 0$ and $\theta \in (0,\pi/2)$, the $\R^+$-supported function $L_{\nu,\gamma,\xi,\theta}$ whose Mellin transform is (\ref{eqn:MTkern}) on the strip of holomorphy $\Ss_{L_{\nu,\gamma,\xi,\theta}}$ (\ref{eqn:SL}), is the density of the random variable $Y = \nu X^\xi$, where $X$ follows the Fisher-Snedecor $F$-distribution with $\frac{2\xi^2}{\gamma^2 \cos^2 \theta}$ and $\frac{2\xi^2}{\gamma^2 \sin^2 \theta}$ degrees of freedom, i.e.\ $X \sim F\left(\frac{2\xi^2}{\gamma^2 \cos^2 \theta},\frac{2\xi^2}{\gamma^2 \sin^2 \theta}\right)$.
\end{proposition}
\begin{proof} See Appendix. \end{proof}

The proof uses the representation of an $F$-distributed random variable as a (rescaled) product of two independent Gamma and Inverse Gamma random variables. This allows us to extend (\ref{eqn:MTkern}) to the cases $\theta = 0$ and $\theta = \pi/2$ as well. Set 
\begin{align} \Ms(L_{\nu,\gamma,\xi,0};z) & = \frac{\nu^{z-1}}{\Gamma\left(\frac{\xi^2}{\gamma^2}\right)} \left(\frac{\gamma^2}{\xi^2} \right)^{\xi(z-1)}  \Gamma\left(\frac{\xi^2}{\gamma^2 }+\xi(z-1) \right), \qquad \Re(z) > 1-\frac{\xi}{\gamma^2}, \label{eqn:MTGamkern} \\ 
\text{ and } \quad  \Ms(L_{\nu,\gamma,\xi,\pi/2};z) & = \frac{\nu^{z-1}}{\Gamma\left(\frac{\xi^2}{\gamma^2}\right)} \left(\frac{\gamma^2}{\xi^2} \right)^{\xi(1-z)}  \Gamma\left(\frac{\xi^2}{\gamma^2 }+\xi(1-z) \right), \qquad \Re(z) < 1+\frac{\xi}{\gamma^2}. \label{eqn:MTInvGamkern}\end{align}
Lemma \ref{prop:scalpow} in Appendix, (\ref{eqn:gamMT}) and (\ref{eqn:invgamMT}) make $L_{\nu,\gamma,\xi,0}$ the density of $\nu X^\xi$ for $X \sim \text{Gamma}\left(\frac{\xi^2}{\gamma^2},\frac{\xi^2}{\gamma^2} \right)$ and $L_{\nu,\gamma,\xi,\pi/2}$ the density of $\nu X^\xi$ for $X \sim \text{InvGamma}\left(\frac{\xi^2}{\gamma^2},\frac{\xi^2}{\gamma^2} \right)$, in agreement with the usual interpretation of the $F$-distribution with infinite (either numerator or denominator) degrees of freedom. %Owing to this, we will not distinguish the cases $\theta = 0$ or $\theta = \pi/2$ from the general case (\ref{eqn:MTkern}) in the sequel. %Comparing (\ref{eqn:MTGamkern}) and (\ref{eqn:MTInvGamkern}) to (\ref{eqn:MTkern}) makes it clear that, in the cases $\lambda = 0$ or $\kappa = 0$, the factors involving either $\lambda$ or $\kappa$ should just be ignored in (\ref{eqn:MTkern}), and we consider that (\ref{eqn:MTkern}) and Proposition \ref{prop:momkern} below hold for $\lambda = 0$ and $\kappa = 0$ as well. The validity of this argument is confirmed by Proposition \ref{prop:asympMell} below, which shows that the effect of $\lambda$ and $\kappa$ is essentially additive as $\lambda, \kappa \to 0$.

\ppn The strip of holomorphy $\Ss_{L_{\nu,\gamma,\xi,\theta}}$ (\ref{eqn:SL}) clarifies how the parameters $\gamma$, $\xi$ and $\theta$ act on the lightness/fatness of the head and the tail of the density $L_{\nu,\gamma,\xi,\theta}$. The ratio $\gamma^2/\xi$ fixes the `overall fatness' of $L_{\nu,\gamma,\xi,\theta}$: the higher the value of $\gamma^2/\xi$, the fatter both its head and its tail. How exactly that overall fatness is shared between the head and the tail of $L_{\nu,\gamma,\xi,\theta}$ is specified by $\theta$: $\theta = 0$ produces a density with as light a tail and as fat a head as can be (given the other parameters), while $\theta = \pi/2$ produces a density with as fat a tail and as light a head as can be (given the other parameters). The balanced case is, of course, for $\theta = \pi/4$. Thus, by playing with the values of $\gamma$, $\xi$ and $\theta$, one can produce a wide variety of different head and tail behaviours for $L_{\nu,\gamma,\xi,\theta}$. Those include exponential behaviours ($\theta = 0$ or $\theta = \pi/2$, or $\xi \to \infty$), and positiveness/unboundedness at $x=0$, for $\xi < \gamma^2 \cos^2\theta$ (see that then, $z = 0 \notin \Ss_f$, meaning that $\E(X^{-1}) = \infty$ and $f(x)$ cannot be $o(1)$ as $x \to 0$).

\ppn 
%The distributions of powers of the $F$-distribution have sometimes been called {\it Generalised $F$-distributions} in the literature, see e.g.\ \cite{Phamgia89}. 
We call a probability density whose Mellin transform can be written under the form (\ref{eqn:MTkern}) (or its simplified versions (\ref{eqn:MTGamkern}) and (\ref{eqn:MTInvGamkern})) for some $\nu,\gamma,\xi > 0$ and $\theta \in [0,\pi/2]$, a {\it Meijer density}. Indeed, all densities defined through (\ref{eqn:MTkern}), having for Mellin transform a rescaled product of two Gamma functions as in (\ref{eqn:MTG11}), are rescaled versions of a $G_{1,1}^{1,1}$-Meijer function. Specifically, for $\theta \in (0,\pi/2)$, it can be checked that 
\begin{align} L_{\nu,\gamma,\xi,\theta}(x) & = \frac{1}{\nu \xi}\frac{1}{\Gamma\left(\frac{\xi^2}{\gamma^2 \cos^2 \theta} \right)\Gamma\left(\frac{\xi^2}{\gamma^2 \sin^2 \theta} \right)}\left(\tan^2 \theta \right)^\xi G_{1,1}^{1,1}\left(\tan^2 \theta \left(\frac{x}{\nu}\right)^{1/\xi} \left| \substack{-\frac{\xi^2}{\gamma^2 \sin^2 \theta}+1-\xi \\  \frac{\xi^2}{\gamma^2 \cos^2 \theta}-\xi}\right.  \right),  \label{eqn:G11} %\\ 
%L_{\nu,\gamma,\xi,0}(x) & = \frac{1}{\nu \xi \Gamma\left(\frac{\xi^2}{\gamma^2} \right)}\left(\frac{\xi^2}{\gamma^2} \right)^\xi G_{0,1}^{1,0}\left(\frac{\xi^2}{\gamma^2}\left(\frac{x}{\nu}\right)^{1/\xi} \left| \substack{- \\  %\frac{\xi^2}{\gamma^2}-\xi}\right.  \right) \notag \\
%L_{\nu,\gamma,\xi,\pi/2}(x) & = \frac{1}{\nu \xi \Gamma\left(\frac{\xi^2}{\gamma^2} \right)}\left(\frac{\gamma^2}{\xi^2} \right)^\xi G_{1,0}^{0,1}\left(\frac{\gamma^2}{\xi^2} \left(\frac{x}{\nu}\right)^{1/\xi} \left| \substack{-\frac{\xi^2}{\gamma^2}+1-\xi \\  -}\right.  \right), \notag
\end{align}
giving explicit forms for those densities in terms of $G_{1,1}^{1,1}$. Similar, appropriately simplified expressions in terms of $G_{1,0}^{0,1}$ and $G_{0,1}^{1,0}$ are valid for $\theta \in \{0,\pi/2\}$ as well. Most of the $\R^+$-supported probability distributions of practical interest are actually Meijer distributions. These include, but are not limited to, the Amoroso/Stacy (i.e., Generalised Gamma), Beta prime, Burr, Chi, Chi-squared, Dagum, Erlang, Fisher-Snedecor, Fr\'echet, Gamma, Generalised Pareto, L\'evy, Log-logistic, Maxwell, Nakagami, Rayleigh, Singh-Maddala and Weibull distributions; see Table \ref{tab:meijerdistr} in Appendix \ref{app:meijerdistr}. All the `inverse' distributions of these, such as the Inverse Gamma, are also Meijer distributions, as it appears clearly from (\ref{eqn:MTkern}) and Corollary \ref{cor:inv} that the class of Meijer distributions is closed under the `inverse' operation. Finally, they admit the log-Normal distribution as limiting case as $\xi \to \infty$. In fact, $L_{\nu,\gamma,\xi,\theta}$ being essentially the density of a certain power of an $F$-distributed random variable, it has strong links to what has sometimes been called the {\it Generalised $F$-distribution} in the literature \citep{Prentice75,McDonal84,Cox08}. %Table \ref{tab:meijerdistr} in Appendix \ref{app:meijerdistr} gives the Meijer parameterisation of those common distributions. 

\begin{remark} It is possible to define probability densities as appropriate rescaled versions of general $G_{p,q}^{m,n}$-functions (\ref{eqn:Gfunc}) for $p+q > 2$ as well. Their Mellin transform would involve a ratio of $p+q$ Gamma factors, as in the full form (\ref{eqn:fullMeijer}). Given the richness of the class of densities defined through $G_{1,1}^{1,1}$, though, the form (\ref{eqn:MTkern}) seems sufficient for many purposes.
\end{remark}

\ppn From (\ref{eqn:MTkern}), the basic properties of densities (\ref{eqn:G11}) are actually easy to obtain. In particular, the coefficient of variation of $L_{\nu,\gamma,\xi,\theta}$ will be of interest in the next sections.

\begin{proposition} \label{prop:momkern} If $\xi > 2\gamma^2 \sin^2 \theta$ and $\theta \in (0,\pi/2)$, the density $L_{\nu,\gamma,\xi,\theta}$ has coefficient of variation $\chi$ equal to
\begin{equation} \chi  = \sqrt{\frac{\Gamma\left(\frac{\xi^2}{\gamma^2 \cos^2 \theta}\right)\Gamma\left(\frac{\xi^2}{\gamma^2 \sin^2 \theta}\right)\Gamma\left(\frac{\xi^2}{\gamma^2 \cos^2 \theta}+2\xi\right)\Gamma\left(\frac{\xi^2}{\gamma^2 \sin^2 \theta}-2\xi\right)}{\Gamma^2\left(\frac{\xi^2}{\gamma^2 \cos^2 \theta}+\xi\right)\Gamma^2\left(\frac{\xi^2}{\gamma^2 \sin^2 \theta}-\xi\right)}-1}. \label{eqn:fullgamma} \end{equation}
\end{proposition}
\begin{proof}  See Appendix. \end{proof}

According to the above observations, if $\theta = 0$ (resp.\ $\theta=\pi/2$) then the factors containing $\sin^2 \theta$ (resp.\ $\cos^2 \theta$) would not appear in (\ref{eqn:fullgamma}). Also, the parameter $\nu$ does not affect the value of $\chi$, which was expected given the purely multiplicative nature of its role as described by Proposition \ref{prop:Meijerkern}.
For $\xi \in \N$, (\ref{eqn:fullgamma}) can be simplified using the recursion formula $\Gamma(\alpha+1) = \alpha\Gamma(\alpha)$ ($-\alpha \notin \N$). For instance, for $\xi = 1$ one finds%, for $\gamma <1$,
\begin{equation} \chi = \gamma \,\sqrt{\frac{1-\gamma^2 \sin^2 \theta \cos^2 \theta}{1-2 \gamma^2 \sin^2 \theta}}. \label{eqn:chigamma1}\end{equation}
Clearly, in this case, $\chi \sim \gamma$ as $\gamma \to 0$. In fact, this asymptotic equivalence holds true in general. 
\begin{proposition} \label{lem:asympCV} Let $\nu,\xi >0$ and $\theta \in [0,\pi/2]$ be fixed, and let $\gamma \to 0$. Then, the coefficient of variation $\chi$ of $L_{\nu,\gamma,\xi,\theta}$ is asymptotically equivalent to $\gamma$: $\chi \sim \gamma$, as $\gamma \to 0$. 
\end{proposition}
\begin{proof}  See Appendix. \end{proof}

\ppn The following asymptotic expansion of (\ref{eqn:MTkern}) as $\gamma \to 0$ will also be of importance in the next sections.

\begin{proposition} \label{prop:asympMell} Consider the Meijer density $L_{\nu,\gamma,\xi,\theta}$ as described above. Fix $\xi>0$ and $\theta \in [0,\pi/2]$. Let $\gamma \to 0$ and $\nu = 1 + \Delta \gamma^2$ for some real constant $\Delta$. Then, the Mellin transform of $L_{\nu,\gamma,\xi,\theta}$ (\ref{eqn:MTkern}) admits the following expansion:
\begin{equation} \Ms(L_{\nu,\gamma,\xi,\theta};z) = 1 + \frac{\gamma^2}{2}\,(z-1) \left(z-1 -\frac{\cos 2\theta}{\xi} +2\Delta \right) + \rho(\gamma,z), \label{eqn:MLasymp} \end{equation}
where $|\rho(\gamma,z)| =O(\gamma^4 (1+|z-1|)^2)$, provided $|z-1|=o(\gamma^{-2})$. 
\end{proposition}
\begin{proof} See Appendix. \end{proof}
\begin{remark}  \label{rem:delta}
The choice $\Delta = \frac{1}{2}\left(1+ \frac{\cos 2\theta}{\xi}\right)$ yields the simple form
\begin{equation} \Ms(L_{\nu,\gamma,\xi,\theta};z) = 1 + \frac{\gamma^2}{2}\,z(z-1) + \rho(\gamma,z), \label{eqn:MLasympsimp} \end{equation} 
which only involves $\gamma$ explicitly.
\end{remark}

\ppn Finally, the Mellin transform of $L^2_{\nu,\gamma,\xi,\theta}$, described by the following result, will also be used.
\begin{proposition} \label{prop:asympMellL2} Consider the Meijer density $L_{\nu,\gamma,\xi,\theta}$ as described above. $(i)$ The Mellin transform of $L^2_{\nu,\gamma,\xi,\theta}$ is 
\begin{multline} \Ms(L^2_{\nu,\gamma,\xi,\theta};z) = \frac{1}{\xi}\,\nu^{z-2}\, \frac{\Bs\left(\frac{2\xi^2}{\gamma^2 \cos^2 \theta},\frac{2\xi^2}{\gamma^2 \sin^2 \theta} \right)}{\Bs^2\left(\frac{\xi^2}{\gamma^2 \cos^2 \theta},\frac{\xi^2}{\gamma^2 \sin^2 \theta} \right)} \\ \times \left(\frac{1}{\tan^2 \theta}\right)^{\xi(z-2)}\,\frac{\Gamma\left(\frac{2\xi^2}{\gamma^2 \cos^2 \theta}+\xi(z-2) \right)}{\Gamma\left(\frac{2\xi^2}{\gamma^2 \cos^2 \theta}\right)} \, \frac{\Gamma\left(\frac{2\xi^2}{\gamma^2 \sin^2 \theta}+\xi(2-z) \right)}{\Gamma\left(\frac{2\xi^2}{\gamma^2 \sin^2 \theta}\right)}, \label{eqn:MTL2} \end{multline}
where $\Bs(\cdot,\cdot)$ is the Beta function, on the strip of holomorphy
\begin{equation} \Ss_{L^2_{\nu,\gamma,\xi,\theta}} = \left\{z \in \C: 2-\frac{2\xi}{\gamma^2 \cos^2 \theta} < \Re(z) < 2+\frac{2\xi}{\gamma^2 \sin^2 \theta}\right\}. \label{eqn:SL2}  \end{equation} 
$(ii)$ Fix $\xi>0$ and $\theta \in [0,\pi/2]$. Let $\gamma \to 0$ and $\nu = 1 + \Delta \gamma^2$ for some real constant $\Delta$. Then we have the asymptotic expansion
\begin{equation}  \Ms(L^2_{\nu,\gamma,\xi,\theta};z) =  \frac{1}{2\sqrt{\pi} \gamma}\left(1+ \omega(\gamma,z)\right), \label{eqn:ML2asymp}\end{equation}
where $|\omega(\gamma,z)| = O(\gamma^2 (1+|z-2|))$, provided $|z-2|=o(\gamma^{-2})$.
\end{proposition}
\begin{proof} See Appendix. \end{proof}

\section{Mellin-Meijer kernel density estimation} \label{sec:Mkde}

\subsection{Basic idea} \label{subsec:basicMKDE}

Following the idea introduced in Section \ref{subsec:LNkde}, we define a `Mellin' version of the kernel estimator of a density $f$ supported on $\R^+$ as 
\begin{equation} \hat{f}_0(x) = (L_\eta \MC \P_n)(x), \label{eqn:MKDEconv} \end{equation}
where $\P_n$ is again the sample empirical measure and $L_\eta$ is an $\R^+$-supported density whose `spread' (to make precise later) is driven by a parameter $\eta >0$ playing the role of smoothing parameter. Obviously, (\ref{eqn:MKDEconv}) is just the multiplicative analog of (\ref{eqn:convolkde}). From (\ref{eqn:Mellinconv}), we have $\hat{f}_0(x) = \int_{0}^\infty L_\eta\left(\frac{x}{v}\right)\frac{1}{n} \sum_{k=1}^n \delta_{X_k}(v)\,\frac{dv}{v}$, that is
\begin{equation} 
\hat{f}_0(x) = \frac{1}{n} \sum_{k=1}^n \frac{1}{X_k} L_\eta\left(\frac{x}{X_k}\right) . \label{eqn:Mellkde}
\end{equation}
Estimator (\ref{eqn:Mellkde}) assesses which observations $X_k$ are local to $x$ through the ratios $x/X_k$, which is natural on $\R^+$. Similarities {\it and} differences between (\ref{eqn:Mellkde}) and expression (2) in \cite{Comte12} and expression (2.4) in \cite{Mnatsakanov12}, are easy to identify. In particular, neither of those two estimators integrates to one, whereas $\hat{f}_0$ always defines a {\it bona fide} density: given that $L_\eta$ is a density, it is obvious that $\hat{f}_0(x) \geq 0 \ \forall x >0$ and $\int_0^\infty \hat{f}_0(x)\,dx =1$. In addition, in contrast to estimators of type (\ref{eqn:asymkde}), estimator (\ref{eqn:Mellkde}) is constructed as a sum of `bumps' $\Lambda^{(k)}_\eta(x) \doteq \frac{1}{X_k} L_\eta\left(\frac{x}{X_k}\right)$, which makes it easy to understand visually -- see \citet[Section 3.1.5]{Hardle04} for related comments in the case of the conventional kernel density estimator (\ref{eqn:convkde}). Figure \ref{fig:bumps} (left panel) illustrates this for an artificial sample of size $n=15$.

\begin{figure}[h]
\centering
\includegraphics[width=0.9\textwidth]{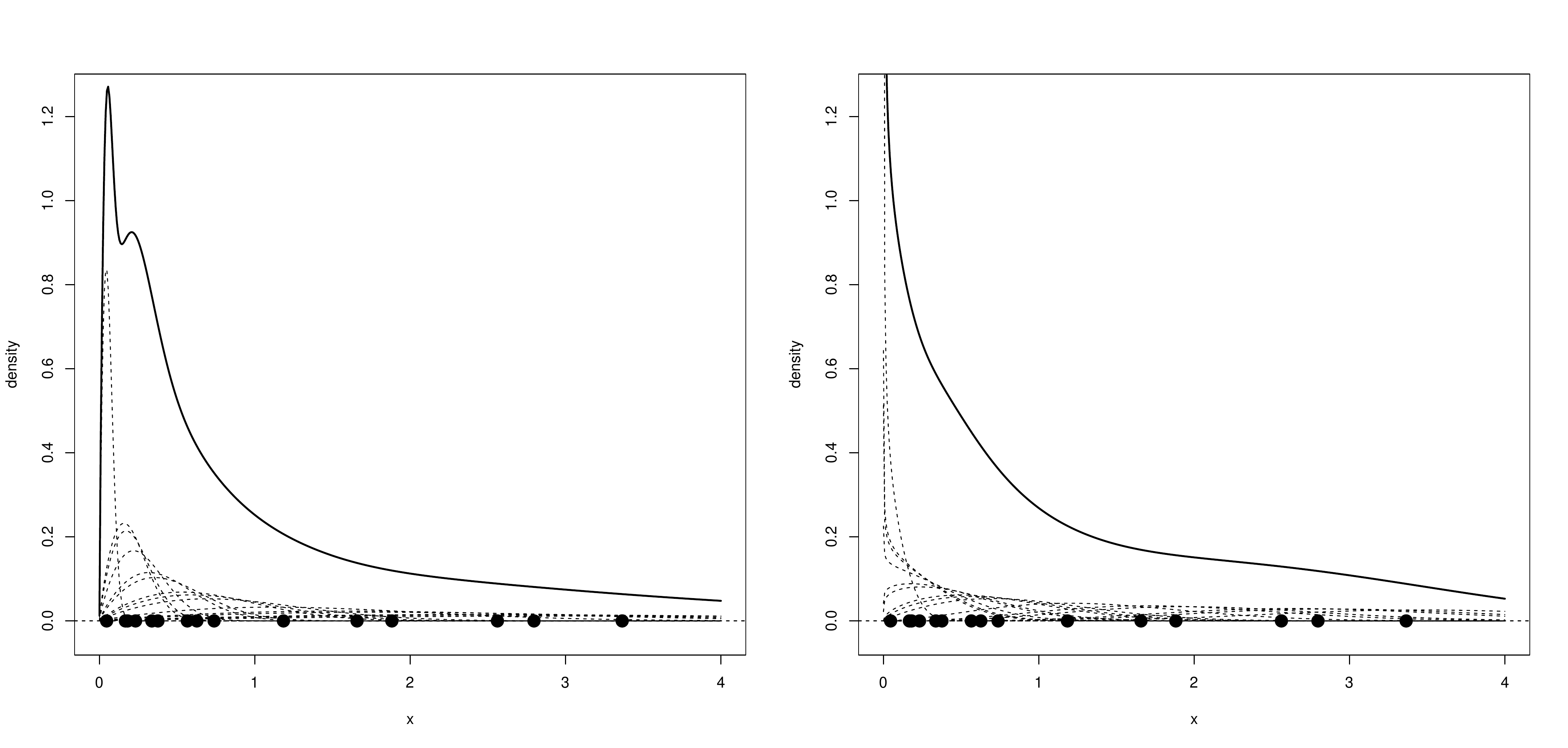}
\caption{Construction of the basic estimator (\ref{eqn:MKDEconv})-(\ref{eqn:Mellkde}) (left panel) and its refined version (\ref{eqn:MKDEconvK})-(\ref{eqn:MKDEK}) with (\ref{eqn:cvk}) (right panel) for an artificial sample of size $n=15$. The observations $X_k$ (big dots) and the associated `bumps' $\Lambda_\eta^{(k)}$ (dashed lines) are shown. The final estimator (thick line) is the sum of those bumps. In both cases, the smoothing parameter is $\eta=0.5$ and the kernel is a Meijer density with $\xi=1/2$ and $\theta = 0$: in (\ref{eqn:Mellkde}), $L_\eta = L_{1+\eta^2,\eta,1/2,0}$; and in (\ref{eqn:MKDEK}), $L_\eta^{(k)}$ is the Meijer kernel $L_{\nu^{(k)}_{\eta},\gamma^{(k)}_{\eta},1/2,0}$, as described in Section \ref{sec:Meijkern}.
}
\label{fig:bumps}
\end{figure}

\ppn Unlike in the conventional case, though, here the `bumps' do not have the same width. By definition, $\Lambda^{(k)}_\eta$ is the density of the random variable $X_k \times \varepsilon$, where $\varepsilon$ has density $L_\eta$ and $X_k$ is fixed. If $\mu_\eta$ and $\sigma_\eta$ are the mean and standard deviation of  $L_\eta$, then $\Lambda^{(k)}_\eta$ has standard deviation $\sigma_\eta^{(k)} = X_k\, \sigma_\eta$, obviously different for each $k$. Hence, to an extent driven by $\eta$, the bumps $\Lambda^{(k)}_\eta$ corresponding to the $X_k$'s close to the boundary 0 are high and narrow, while those in the right tail are wide and flat (Figure \ref{fig:bumps}, left). More smoothing is thus automatically applied in the tail than close to the boundary (`{\it adaptive behaviour}'), and this is essentially how both boundary issues and the `spurious bumps' problem are addressed by (\ref{eqn:MKDEconv})-(\ref{eqn:Mellkde}).

\ppn Given that $\Lambda_\eta^{(k)}$ has mean $\mu_\eta^{(k)} = X_k\, \mu_\eta$, what is common to all the $\Lambda^{(k)}_\eta$'s is actually their {\it coefficient of variation} $\chi^{(k)}_\eta \doteq \frac{\sigma_\eta^{(k)}}{\mu_\eta^{(k)}} = \frac{\sigma_\eta}{\mu_\eta} \doteq \chi_\eta$. Note that this is also the coefficient of variation of $L_\eta$, the `canonical' bump for $X_k=1$. This points out the natural role of the coefficient of variation of the kernel $L_\eta$ in this framework, suggesting to defining it as the global smoothing parameter $\eta$.

\subsection{Boundary undersmoothing, tail oversmoothing and smoothing transfer} \label{subsec:BUTO}

Unfortunately, this idea has to be slightly amended. The reason is that the seemingly desirable `adaptive behaviour' described in the previous section actually occurs in excess. Given that $\sigma_\eta^{(k)} = X_k\, \sigma_\eta\simeq 0$ for $X_k \simeq 0$, the effective amount of smoothing applied in the boundary area is virtually nil. As a result, the basic estimator (\ref{eqn:Mellkde}) severely undersmooths close to 0 (regardless of the value of $\eta$), and typically shows a very rough and erratic behaviour there (Figure \ref{fig:bumps}, left). On the other hand, the effective amount of smoothing applied in the tail area is very high, as $\sigma_\eta^{(k)}$ gets huge as $X_k$ does. Visually less disturbing, this severe oversmoothing effect usually goes unnoticed, but it exists nonetheless.

\ppn Theoretically, this `{\it boundary undersmoothing/tail oversmoothing}' behaviour materialises through a factor $1/x$ in the asymptotic variance expression -- see e.g.\ (\ref{eqn:asymptvarlognormkde}), recalling that the log-Normal kernel density estimator (\ref{eqn:lognormkde}) is a particular case of (\ref{eqn:Mellkde}) with $L_\eta$ being a Log-Normal density.\footnote{The same factor $1/x$ appears in variance expressions of many other estimators, including \cite{Comte12}'s, \cite{Jin03}'s, \cite{Marchant13}'s and \cite{Mnatsakanov12}'s.} In fact, this arises because the Haar measure on $(\R^+,\times)$ is $\nu(dx) = \frac{dx}{x}$. It has thus a mere topological origin, and cannot really be thought of as a deficiency of the idea: the estimator does exactly what it is supposed to do in $(\R^+,\times)$. It remains, though, that the so-produced estimates are visually unsatisfactory, which calls for some adjustment. What is suggested here is a very natural (as opposed to an {\it ad-hoc} correction) {\it smoothing transfer} operation: given that the basic estimator smooths too much in the tail and not enough at the boundary, make it use some of the amount of smoothing in excess in the tail for filling the shortage of smoothing at the boundary. 

\ppn This transfer is easily achieved by making the coefficient of variation $\chi_\eta^{(k)}$ of $\Lambda^{(k)}_\eta$ a decreasing function of $X_k$, instead of keeping it constant for all $k$. One can think of setting $\chi_\eta^{(k)} = \eta/X_k$, with $\eta >0$ some smoothing parameter. This enforces $\sigma_\eta^{(k)} \equiv \eta \mu_\eta$ for all $k$, i.e., all the `bumps' have the same width (same standard deviation). Of course, $\chi_\eta^{(k)}  \sim 1/X_k$ means that one exactly adapts to the `Haar geometry' of $(\R^+,\times)$. As a result, the whole extra amount of smoothing initially applied in the tail is transferred back to the boundary area, and the exact same level of smoothing is applied all over $\R^+$. This means, however, that the adaptive behaviour of the initial estimator has been destroyed entirely, which typically implies the resurgence of boundary issues (bias) and spurious bumps in the tail. This is not desirable.

\ppn A natural trade-off between `full adaptation' and `no adaptation' to the Haar geometry is achieved by setting $\chi_\eta^{(k)} = \eta/\sqrt{X_k}$ or
\begin{equation} \chi_\eta^{(k)} = \frac{\eta}{\sqrt{\eta^2+X_k}}, \label{eqn:cvk} \end{equation}
for some smoothing parameter $\eta>0$. Both choices produce equivalent estimators asymptotically ($\eta \to 0$ as $n \to \infty$, see Assumption \ref{ass:eta} below), however (\ref{eqn:cvk}) typically produces more stable estimates in practice, as $\chi_\eta^{(k)}$ stays bounded for $X_k \simeq 0$. Hence (\ref{eqn:cvk}) will be the preferred option throughout the rest of the paper. In any case, now $\sigma_\eta^{(k)} \sim \sqrt{X_k}$, and the bumps remain wider in the tail than at the boundary. The estimator keeps adapting to the boundary and tail areas, but not as excessively as previously. This is illustrated in Figure \ref{fig:bumps} (right panel) for the same sample and with the same smoothing parameter $\eta$ as in the left panel: the `bumps' at the boundary are no more as narrow, and the bumps in the tail no more as flat, as in the initial case. The final estimate of $f$ seems rightly smooth all over $\R^+$. 

\ppn Figure \ref{fig:bumps} highlights another major benefit of allowing $\chi_\eta^{(k)}$ to depend on $X_k$: now the bumps vary in shape as well. In the basic case (\ref{eqn:Mellkde}), all the bumps $\Lambda^{(k)}_\eta(x) \doteq \frac{1}{X_k} L_\eta\left(\frac{x}{X_k}\right)$ are just rescaled versions of $L_\eta$, hence have the same shape. In particular, if $L_\eta$ is such that $L_\eta(0) = 0$, then the basic estimator $\hat{f}_0(0) \equiv 0$ automatically, no matter what the data show in the boundary area (Figure \ref{fig:bumps}, left).\footnote{This problem of `$\hat{f}(0) \equiv 0$' is shared by many other kernel estimators of $\R^+$-supported densities, including \cite{Jin03}'s, \cite{Scaillet04}'s, \cite{Marchant13}'s and \cite{Mnatsakanov12}'s.} This is no more the case when $\Lambda^{(k)}_\eta$ may have different coefficients of variation. For instance, in Figure \ref{fig:bumps} (right), the bumps associated to the data close to 0 are no more tied down to 0: as their coefficient of variation increases, they are forced to climb along the $y$-axis. This enables the final estimator to take a positive and even infinite value at $x=0$.

\ppn This justifies to define a `refined' version of estimator (\ref{eqn:MKDEconv}) as
\begin{equation} \hat{f}(x) = \frac{1}{n} \sum_{k=1}^n (L^{(k)}_\eta \MC \delta_{X_k})(x), \label{eqn:MKDEconvK} \end{equation}
where each $L^{(k)}_\eta$ is an $\R^+$-supported density whose coefficient of variation $\chi_\eta^{(k)}$ depends on both an overall smoothing parameter $\eta$ and on $X_k$ through (\ref{eqn:cvk}). Explicitly, this is
\begin{equation}
 \hat{f}(x) = \frac{1}{n} \sum_{k=1}^n \frac{1}{X_k} L_\eta^{(k)}\left(\frac{x}{X_k}\right). \label{eqn:MKDEK}
\end{equation}
The observations made about (\ref{eqn:Mellkde}) obviously remain valid: $\hat{f}$ is always a valid probability density, it assesses the proximity between $x$ and $X_k$ through their ratio $x/X_k$, etc. Actually, (\ref{eqn:MKDEK}) is just another version of (\ref{eqn:Mellkde}) which allows a better allocation of the total amount of smoothing applied, but without affecting the intuitive simplicity of the basic estimator.

\begin{remark} Estimator (\ref{eqn:MKDEK}) may be thought of as a `sample-smoothing' kernel estimator \citep{Terrell92}, in the sense that the smoothing parameter associated with a particular `bump', here $\chi_\eta^{(k)}$, varies with $X_k$. However, conventional `sample-smoothing' aims to produce adaptive estimators by using a large bandwidth where data are sparse, i.e.\ when the density $f$ is low, and a small bandwidth where data are abundant, i.e.\ where $f$ is high, see e.g.\ \cite{Abramson82}'s square-root law. As a result, it typically requires pilot estimation of $f$, which is not without causing further issues \citep{Terrell92,Hall95}. Here, `sample-smoothing' is deterministically articulated around (\ref{eqn:cvk}) -- obviously, no pilot estimation is necessary -- and just aims to adjust to the particular `Haar geometry' of $(\R^+,\times)$.
\end{remark}

\subsection{Meijer kernels} \label{sec:Meijkern}

Owing to the connection between Mellin transforms and Meijer densities expounded in Section \ref{sec:motiv}, using kernels of Meijer type seems natural here. Hence, consider taking for $L_\eta^{(k)}$ in (\ref{eqn:MKDEconvK})-(\ref{eqn:MKDEK}) a Meijer density $L_{\nu,\gamma,\xi,\theta}$ as described in Section \ref{subsec:Meijerdens}. Fix the parameters $\xi>0$ and $\theta \in [0,\pi/2]$ and keep them constant for all $k$. Those essentially determine the type of kernels that will be used. E.g., if we take $\theta = 0$ and $\xi=1$, then the $L_\eta^{(k)}$'s are Gamma densities; with $\theta = \pi/2$ and $\xi =1$, then the $L_\eta^{(k)}$'s are Inverse Gamma densities; with $\xi=1/2$ and $\theta=0$, the $L_\eta^{(k)}$'s are Nakagami densities, etc.; refer to Table \ref{tab:meijerdistr} in Appendix \ref{app:meijerdistr}. In some sense, chosing $\xi$ and $\theta$ is akin to chosing which kernel (e.g.\ Epanechnikov or Gaussian) to work with in the conventional case (\ref{eqn:convkde}). Here, though, this choice should be made more carefully, as $\theta$ and $\xi$ drive the head and tail behaviour of the kernels (Section \ref{subsec:Meijerdens}). If it is anticipated that $f$ has a fat tail and/or head, then it seems intuitively clear that working with kernels $L^{(k)}_\eta$ with fat tails and/or heads should be beneficial to the estimator, and $\xi$ and $\theta$ should be picked accordingly. This will be confirmed by theoretical considerations in Section \ref{subsec:asympt} (in particular, see Assumption \ref{ass:Lxi}).

\ppn Now, for $\eta >0$ some smoothing parameter, set 
\begin{equation} 
 \gamma   \doteq \gamma_\eta^{(k)} = \frac{\eta}{\sqrt{\eta^2 + X_k}}. \label{eqn:coefgamma} \\
 \end{equation}
This is motivated by (\ref{eqn:cvk}) and Proposition \ref{lem:asympCV}, stating that asymptotically the coefficient of variation of $L_{\nu,\gamma,\xi,\theta}$ is equivalent to $\gamma$. Finally, motivated by Remark \ref{rem:delta}, set
\begin{equation} \nu  \doteq \nu_\eta^{(k)} = 1+ \frac{1}{2} \,\left.\gamma_\eta^{(k)}\right.^2 \left(1+ \frac{\cos 2\theta}{\xi}\right). \label{eqn:nucoef}\end{equation}
This choice ensures that the expectation of $L_\eta^{(k)}$ is asymptotically equivalent to $1 +\left.\gamma_\eta^{(k)}\right.^2$ for all values of $\theta$ and $\xi$ (take $z=2$ in (\ref{eqn:MLasympsimp})). Intuitively, given that $L_\eta^{(k)}$ is the density of the multiplicative noise used to dilute $X_k$ (see Section \ref{subsec:LNkde}), it is understood that $L_\eta^{(k)}$ should have expectation (close to) 1 for all $k$ and $\eta$. The reason why it should not be {\it exactly} 1, is again that expectations are slightly distorted in the multiplicative environment. For instance, the expectation of the log-Normal$(0,\sigma^2)$ distribution is not $1$, but $\exp(\sigma^2/2) \sim (1+\sigma^2/2)$ as $\sigma^2 \to 0$. The parameter $\nu$ as in (\ref{eqn:nucoef}) exactly accounts for that distortion.

\ppn We call kernels $L^{(k)}_\eta = L_{\nu^{(k)}_{\eta},\gamma^{(k)}_{\eta},\xi,\theta}$ with this parameterisation (\ref{eqn:coefgamma})-(\ref{eqn:nucoef}), {\it Meijer kernels}. Figure \ref{fig:kernels} shows examples of such Meijer kernels for $\xi \in \{1/2, 1,2\}$ and $\theta \in \{0,\pi/4,\pi/2\}$, for $X_k=1$ (canonical kernel $L_\eta$) and several values of the smoothing parameter $\eta$. As $\eta$ approaches 0, the kernels concentrate around 1 with a fading effect of the values of $\theta$ and $\xi$ on their shape, as suggested by expansion (\ref{eqn:MLasympsimp}).

% \begin{align} 
% \Ms(L_\eta^{(k)};z) & = 1+ \frac{1}{2} \frac{\eta^2}{X_k} z(z-1)  + P_k(\eta,z), \label{eqn:asympexpMLk} \\
% \Ms({L_\eta^{(k)}}^2;z) & = \frac{1}{2\sqrt{\pi}}\,\frac{\sqrt{X_k}}{\eta} + \Omega_k(\eta,z), \label{eqn:asympexpML2k}
% \end{align}
% where $|P_k(\eta,z)|=O()$ and $|\Omega(\eta,z)|=O()$. These expansions will play a major role when deriving the asymptotic properties of the estimator in the next section. 

\begin{figure}[h]
\centering
\includegraphics[width=\textwidth]{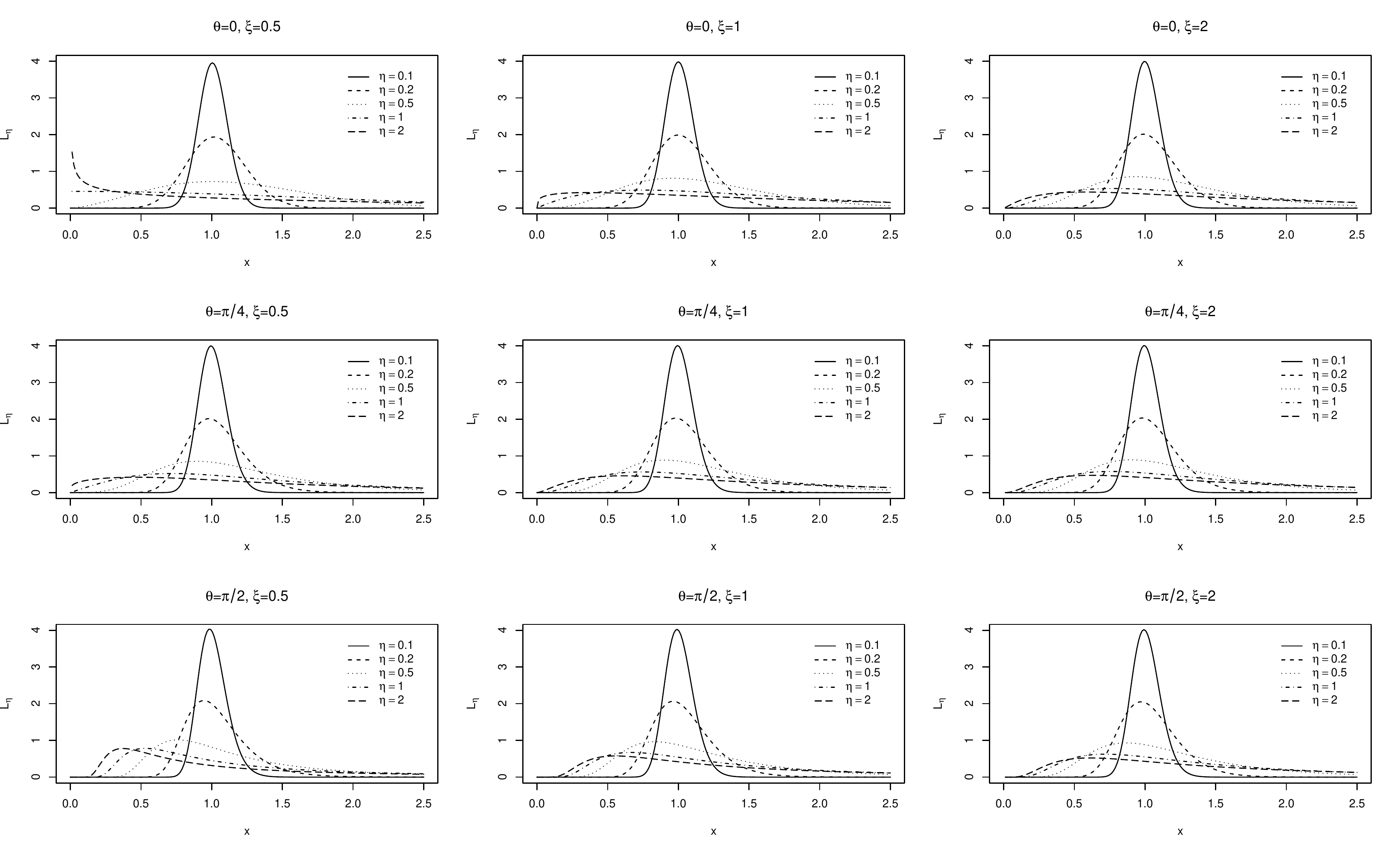}
\caption{Meijer kernels $L^{(k)}_\eta = L_{\nu^{(k)}_{\eta},\gamma^{(k)}_{\eta},\xi,\theta}$ for $\theta = 0, \pi/4, \pi/2$ and $\xi = 1/2, 1, 2$ for $X_k = 1$.}
\label{fig:kernels}
\end{figure}

\begin{remark}
As mentioned above, setting $\theta = 0$ and $\xi=1$ as parameters of the Meijer kernels amounts to using Gamma densities for $L_\eta^{(k)}$ in (\ref{eqn:MKDEK}). However, the so-defined estimator is obviously {\it not}, and is not even related to, \cite{Chen00}'s `Gamma kernel estimator' (in any of its forms). Indeed, with $\theta = 0$ and $\xi=1$, (\ref{eqn:MTGamkern}) tells that $L_\eta^{(k)}$ is the density of $\nu^{(k)}_{\eta} X$, where $X \sim \Gamma(1/{\gamma_\eta^{(k)}}^2,1/{\gamma_\eta^{(k)}}^2)$, that is, with (\ref{eqn:coefgamma})-(\ref{eqn:nucoef}), $L_\eta^{(k)}$ is the density of the $\Gamma\left(\frac{X_k+\eta^2}{\eta^2},\frac{(X_k+\eta^2)^2}{\eta^2(2\eta^2+X_k)}\right)$-distribution. This gives for the final estimator the explicit form
\begin{equation} \hat{f}(x) = \frac{1}{n} \sum_{k=1}^n \frac{1}{\Gamma(1+X_k/\eta^2)} \left(\frac{(\eta^2+X_k)^2}{\eta^2 X_k(2\eta^2+X_k)} \right)^{1+X_k/\eta^2} x^{X_k/\eta^2} e^{-\frac{(\eta^2+X_k)^2}{\eta^2 X_k(2\eta^2+X_k)}\,x}.\label{eqn:explicitgammaMMkde} \end{equation}
This cannot be compared to (\ref{eqn:gammakde}). The slightly more complicated nature of expression (\ref{eqn:explicitgammaMMkde}) should not be repelling as it is given here just for illustration. Exactly as all the properties of the conventional estimator are obtained from the generic expression (\ref{eqn:convkde}) only, without the need to write explicitly the kernel function $K$, the behaviour of the Mellin-Meijer kernel density estimator is to be comprehended exclusively through (\ref{eqn:MKDEK}) and the general properties of Meijer kernels.
\end{remark}

\section{Asymptotic properties} \label{subsec:asympt}

In this section we obtain the asymptotic properties of the Mellin kernel density estimator (\ref{eqn:MKDEK}) with Meijer kernels parameterised as described in Section \ref{sec:Meijkern}, under the following assumptions.

\begin{assumption} The sample $\{X_k,k =1,\ldots,n\}$ consists of i.i.d.\ replications of a positive random variable $X$ whose distribution $F$ admits a density $f$ twice continuously differentiable on $(0,\infty)$; \label{ass:iid} \end{assumption} 
\begin{assumption} There exist $\alpha, \beta \in (0,+\infty]$, with $2\alpha + \beta > 5/2$, such that $\E(X^{-\alpha}) < \infty$ and $\E(X^{\beta}) < \infty$;\label{ass:Sf}\end{assumption} 
\begin{assumption} For any smoothing parameter $\eta >0$ and for each $k \in \{1,\ldots,n\}$, the kernel $L^{(k)}_\eta$ is a Meijer kernel as described in Section \ref{sec:Meijkern}, with $\xi >0$ and $\theta \in [0,\pi/2]$ such that $\frac{\xi}{\cos^2 \theta} > \frac{1}{4} -\frac{\beta}{2}$ and $\frac{\xi}{\sin^2 \theta} > 1 - \alpha$; \label{ass:Lxi} \end{assumption}
\begin{assumption} The smoothing parameter $\eta \doteq \eta_n$ is such that $\eta \to 0$ and $n\eta \to \infty$ as $n \to \infty$. \label{ass:eta}\end{assumption}

Assumption \ref{ass:iid} fixes the setup. The requirement that $f$ has two continuous derivatives is classical in kernel estimation. Assumption \ref{ass:Sf} is a condition on the existence of some negative and positive moments of $f$, and excludes densities with {\it both} `very' fat head and tail, but it is actually a very mild requirement. In particular, $f$ is allowed to be positive and even unbounded at the boundary $x=0$ ($\alpha <1$), and/or to have power law decay in its tail ($\beta < \infty$), provided that it does not show extreme versions of those behaviours simultaneously ($\alpha$ and $\beta$ cannot be both very small at the same time). Assumption \ref{ass:Lxi} essentially requires that the parameters $\xi$ and $\theta$ of the Meijer kernels enable the estimator to properly reconstruct the head and tail behaviour of $f$. For instance, for $\alpha \simeq 0$ ($f$ has a very fat head), it would not work to take $\xi$ `small' and $\theta \simeq \pi/2$ (lightest head for the kernel, see Figure \ref{fig:kernels}). The imposed conditions leave much freedom about the choice of $\xi$ and $\theta$, though, and are restrictive only in extreme cases. For instance, only for $\beta < 1/2$ ({\it extremely} fat tail for $f$) would the condition $\frac{\xi}{\cos^2 \theta} > \frac{1}{4} -\frac{\beta}{2}$ not be trivially satisfied. Finally, Assumption \ref{ass:eta} is the classical condition on the smoothing parameter in kernel estimation. Under these assumptions, we have the following result.

\begin{theorem} \label{thm:consistMISE} Under Assumptions \ref{ass:iid}-\ref{ass:eta}, the Mellin-Meijer kernel density estimator $\hat{f}(x)$ (\ref{eqn:MKDEconvK})-(\ref{eqn:MKDEK}) is such that
\begin{equation} \E\left(\int_0^\infty x^{2c-1} \,\left(\hat{f}(x) - f(x)\right)^2\,dx \right) = O(\eta^4) +O((n\eta)^{-1}) \qquad \text{ as } n \to \infty, \label{eqn:WMISE} \end{equation}
for
\begin{equation} 
 c \in \left( \max\left(2-\alpha,\frac{3-2\alpha}{4},1-\frac{\xi}{\cos^2 \theta}\right),\min\left(\frac{3+2\beta}{4},1+\frac{\xi}{\sin^2 \theta} \right)\right). \label{eqn:ccond}
\end{equation}
\end{theorem}
\begin{proof} See Appendix. \end{proof}
Note that, under Assumptions \ref{ass:Sf} and \ref{ass:Lxi}, (\ref{eqn:ccond}) defines a nonempty interval. Theorem \ref{thm:consistMISE} establishes the convergence to 0 of a weighted Mean Integrated Squared Error (MISE) of the estimator, where the set of values $c$ defining weights $x^{2c-1}$ assuring convergence essentially depends on the assumed negative and positive moments of $f$. The consistency of the estimator in the usual MISE- ($L_2$-)sense follows.

\begin{corollary} \label{cor:optrate} Under Assumptions \ref{ass:iid}-\ref{ass:eta}, with $\alpha > 3/2$ in Assumption \ref{ass:Sf} and $\xi/\cos^2 \theta > 1/2$ in Assumption \ref{ass:Lxi}, the Mellin-Meijer kernel density estimator $\hat{f}(x)$ (\ref{eqn:MKDEconvK})-(\ref{eqn:MKDEK}) is such that
\begin{equation} \E\left(\int_0^\infty \left(\hat{f}(x) - f(x)\right)^2\,dx \right) = O(\eta^4) +O((n\eta)^{-1}) \to 0 \qquad \text{ as } n \to \infty. \label{eqn:optMISE} \end{equation}
\end{corollary}
\begin{proof} If $\alpha > 3/2$ and $\xi/\cos^2 \theta > 1/2$, then $c=1/2$ belongs to (\ref{eqn:ccond}). \end{proof}
Of course, the optimal rate of convergence is achieved for $\eta \sim n^{-1/5}$, which gives 
\[\E\left(\int_0^\infty \left(\hat{f}(x) - f(x)\right)^2\,dx \right) = O(n^{-4/5}), \]
the usual optimal rate of convergence for nonparametric estimation of a univariate probability density under Assumption \ref{ass:iid}. Note that Corollary \ref{cor:optrate} establishes this result for $\alpha > 3/2$, that is, $\E(X^{-3/2}) < \infty$. This requires $f(x) = o(\sqrt{x})$ as $x \to 0$, and in particular, $f(0) = 0$. This is restrictive: for instance, \cite{Chen00} showed that his `modified' Gamma kernel estimator has pointwise (i.e., at some fixed $x$) bias proportional to $xf''(x)$, and pointwise variance proportional to $x^{-1/2}f(x)$, hence deduced its MISE-consistency under the assumptions that $(i)$ $\int (xf''(x))^2\,dx < \infty$ and $(ii)$ $\int x^{-1/2}f(x)\,dx < \infty$. If $(ii)$ is obviously equivalent to $\alpha > 1/2$ in our framework, $(i)$ does not tell much about the moments of $X$. If $\E(X^{-3/2})<\infty$ implies $xf''(x) = o(x^{-1/2})$ as $x \to 0$, there exist distributions with $\int (xf''(x))^2\,dx < \infty$ but with $\E(X^{-3/2}) = \infty$. Those include the Exponential distribution, to cite only one simple example.

\ppn The results presented thus far naturally followed from the properties of the Mellin transform of $f$ {\it inside its strip of holomorphy}. As such, (\ref{eqn:WMISE})-(\ref{eqn:optMISE}) have been proved under proper conditions on the existence of moments of $X$ only, as those define $\Ss_f$. However, those assumptions can actually be relaxed to conditions similar to \cite{Chen00}'s. The condition $c > 2-\alpha$ in (\ref{eqn:ccond}), which requires $\alpha > 3/2$ if one wants $c=1/2$, comes from resorting to the general identity 
\begin{equation} \Ms(xf''(x);z) = z(z-1)\Ms(f;z-1), \qquad \text{ for } z-1 \in \Ss_f \label{eqn:Mxfsec} \end{equation}
in the proof of Theorem \ref{thm:consistMISE}. The point is that `$z-1 \in \Ss_f$' is a sufficient condition for (\ref{eqn:Mxfsec}) to be true, but is not necessary. The simple case of the Exponential density heuristically illustrates the claim. Indeed $f(x) = e^{-x} \indic{x>0}$ has Mellin transform $\Ms(f;z) = \Gamma(z)$ on $\Ss_f = (0,\infty)$. Inverting (\ref{eqn:Mxfsec}) through (\ref{eqn:invMellin}) one has
\begin{equation} xe^{-x} = xf''(x) = \frac{1}{2\pi i} \int_{\Re(z)=c} x^{-z} z(z-1)\Ms(f;z-1) \,dz = \frac{1}{x}\,\frac{1}{2\pi i} \int_{\Re(z)=c-1} x^{-z} z(z+1)\Ms(f;z) \,dz, \label{eqn:xfsec}\end{equation}
for $c-1 \in S_f$, that is, for $c > 1$. However, from (\ref{eqn:CSGam}) we know that, for $\Re(z)\in (-1,0)$, $\Ms(f;z) = \Gamma(z)$ is the Mellin transform of $f_1(x) = e^{-x}+1$, which allows to write $xf_1''(x) = \frac{1}{x}\,\frac{1}{2\pi i} \int_{\Re(z)=c-1} x^{-z} z(z+1)\Ms(f;z) \,dz$ for $c \in (0,1)$. As evidently $f_1'' \equiv f''$, it follows that (\ref{eqn:xfsec}) holds true for $c>0$, and actually, for $c>-1$, as $f''_2 \equiv f''$ as well. 
 
\ppn What theoretically validates this argument is that $\Ms(f;z) = \Gamma(z)$, initially defined as $\int_0^\infty x^{z-1}e^{-x}\,dx$ for $\Re(z) > 0$ only, can be analytically continued on the negative half of the real axis by making repeated use of the identity $\Gamma(z) = (z-1)\Gamma(z-1)$ ($-\Re(z) \notin \N$). This implies that $\Ms(f;z-1)$ in (\ref{eqn:Mxfsec}) is well defined even for $z-1 \notin \Ss_f$. So, more generally, assuming that the {\it analytical continuation} of $\Ms(f;z)$ to the left of $\Ss_f$ is well-behaved in some sense, the proof of Theorem \ref{thm:consistMISE} carries over when relaxing the condition $c >2-\alpha$, and (\ref{eqn:WMISE}) remains valid. Actually, the analytic continuation of a Mellin transform $\Ms(f;z)$ outside its strip of holomorphy gives a much more complete picture of the behaviour of $f$ at 0 or at $\infty$, see \citet[p.86 ff.]{Paris01} or \citet[Theorem 5]{Wong01} for details. Thorough discussion of this aspect is beyond the scope of this paper, nevertheless we can state the following result.

\begin{proposition} \label{cor:optrate2} Under Assumptions \ref{ass:iid}-\ref{ass:eta}, with $\alpha > 1/2$ in Assumption \ref{ass:Sf} and $\xi/\cos^2 \theta > 1/2$ in Assumption \ref{ass:Lxi}, the Mellin-Meijer kernel density estimator $\hat{f}(x)$ (\ref{eqn:MKDEconvK})-(\ref{eqn:MKDEK}) is such that
\begin{equation} \E\left(\int_0^\infty \left(\hat{f}(x) - f(x)\right)^2\,dx \right) = O(\eta^4) +O((n\eta)^{-1}) \to 0 \qquad \text{ as } n \to \infty, \label{eqn:optMISE2} \end{equation}
provided that $\int_0^\infty (x f''(x))^2\,dx < \infty$.
\end{proposition}
\begin{proof} See Appendix. \end{proof}

\ppn Further support for Proposition \ref{cor:optrate2} is brought by analysing the pointwise properties of the Mellin-Meijer kernel density estimator.

\begin{theorem} \label{thm:bias-var} Under Assumptions \ref{ass:iid}-\ref{ass:eta}, the Mellin-Meijer kernel density estimator $\hat{f}(x)$ (\ref{eqn:MKDEconvK})-(\ref{eqn:MKDEK}) at any $x \in (0,\infty)$ has asymptotic pointwise bias and variance given by 
\begin{align}
 \bias\left(\hat{f}(x)\right) & = \frac{1}{2}\eta^2 x f''(x) + o(\eta^2) \label{eqn:asympbias} \\
\var\left( \hat{f}(x)\right) & = \frac{f(x)}{2\sqrt{\pi} n\eta\sqrt{x}} +o((n\eta)^{-1}), \label{eqn:asympvar}
\end{align}
as $n \to \infty$.
\end{theorem}
\begin{proof} See Appendix. \end{proof}

\ppn Theorem \ref{thm:bias-var} allows one to write the Asymptotic (i.e., only dominant terms) Mean Integrated Squared Error (AMISE) of estimator (\ref{eqn:MKDEK}) as
\begin{equation} \text{AMISE}(\hat{f}) = \frac{1}{4} \eta^4 \int_0^\infty (xf''(x))^2\,dx + \frac{1}{2\sqrt{\pi}} \frac{1}{n\eta} \int_0^\infty x^{-1/2} f(x)\,dx,  \label{eqn:AMISE} \end{equation}
which is indeed $O(\eta^4) +O((n\eta)^{-1})$ provided that $\int_0^\infty x^{-1/2} f(x)\,dx<\infty$ (i.e., $\alpha > 1/2$) and $\int_0^\infty (xf''(x))^2\,dx < \infty$. Interestingly, (\ref{eqn:asympbias}) and (\ref{eqn:asympvar}) are the same as the asymptotic expressions for the `away-from-the-boundary' bias and variance of \cite{Chen00}'s modified Gamma kernel estimator (which admits a different behaviour in the boundary area as it is manually modified there). The (asymptotic) bias (\ref{eqn:asympbias}) only depends on $f''$ (not on $f$ or $f'$), as opposed to (\ref{eqn:asymptbiaslognormkde}) or the bias of \cite{Chen00}'s original Gamma estimator and modified Gamma estimator at the boundary. The (asymptotic) variance (\ref{eqn:asympvar}) is proportional to $1/\sqrt{x}$. This factor obviously remains unbounded as $x \to 0$ in theory, but tends to $\infty$ much slower than the factor $1/x$ in (\ref{eqn:asymptvarlognormkde}). Note that, if we exactly adapted to the `Haar geometry' of $(\R^+,\times)$ by taking $\chi_\eta^{(k)}  \sim 1/X_k$ for the Meijer kernels as briefly contemplated in Section \ref{subsec:BUTO}, we would produce an estimator such that $\bias\left(\hat{f}(x)\right) \sim \frac{1}{2}\eta^2 f''(x)$ and $\var\left( \hat{f}(x)\right) \sim \frac{f(x)}{2\sqrt{\pi} n\eta}$, the usual pointwise bias and variance asymptotic expressions for the conventional estimation (\ref{eqn:convkde}). For the reasons expounded in Section \ref{subsec:BUTO}, this might not be desirable in this framework, though.

\section{Smoothing parameter selection} \label{sec:smoothpar}

The choice of the smoothing parameter for acurate nonparametric estimation is always crucial, and estimator (\ref{eqn:MKDEK}) is no exception. Hence a data-driven way of selecting a reasonable value of $\eta$ is desirable. Although the classical idea of cross-validation can be easily adapted to this case, Mellin transform ideas again provide a natural framework for deriving an easy {\it plug-in} selector. 

\ppn From (\ref{eqn:WMISE}), (\ref{eqn:asympbias}) and (\ref{eqn:asympvar}), one can write the asymptotic WMISE of the estimator as 
\begin{equation} \text{AWMISE}(\hat{f}) = \frac{1}{4} \eta^4 \int_0^\infty x^{2c+1}f''^2(x)\,dx + \frac{1}{2\sqrt{\pi}} \frac{1}{n\eta} \int_0^\infty x^{2c-3/2} f(x)\,dx,  \label{eqn:AWMISE} \end{equation}
for any $c$ in (\ref{eqn:ccond}). Balancing the two terms, one easily obtains the asymptotically optimal value of $\eta$:
\begin{equation} \eta_{\text{opt},c} = \left(\frac{(2\sqrt{\pi})^{-1} \int_0^\infty x^{2c-3/2}f(x)\,dx}{\int_0^\infty x^{2c+1}f''^2(x)\,dx} \right)^{1/5} n^{-1/5}. \label{eqn:etaopt} \end{equation}
Plug-in methods attempt to estimate the unknown factors in (\ref{eqn:etaopt}) in order to produce an approximation of $\eta_{\text{opt},c}$. If estimating $ \int_0^\infty x^{2c-3/2}f(x)\,dx = \E(X^{2c-3/2})$ by $\frac{1}{n} \sum_{k=1}^n X_k^{2c-3/2}$ is straightforward, estimating the denominator involving $f''$ is less obvious. For conventional kernel estimation, this step usually requires estimating higher derivatives of $f$, which in turn requires the selection of pilot smoothing parameters and/or resorting to a `reference distribution', see e.g.\ \cite{Sheather91}.

\ppn Here, combining (\ref{eqn:Mxfsec}) and (\ref{eqn:Parseval}) yields
\begin{equation} \int_0^\infty x^{2c+1}f''^2(x)\,dx = \frac{1}{2\pi} \int_{\Re(z)=c}  \left|z(z-1) \Ms(f;z-1) \right|^2\,dz, \label{eqn:intbias2}\end{equation}
for $c -1 \in \Ss_f$. Of course, $\Ms(f;z-1) = \E(X^{z-2})$ can be naturally estimated by $\Ms(\P_n;z-1)=n^{-1} \sum_{k=1}^n X_k^{z-2}$. Now, if $z=c + i \omega$, $|z(z-1)|^2 = (c(c-1)-\omega^2)^2+(2c-1)^2 \omega^2$, and
\begin{align*}
\left|\sum_{k=1}^n X_k^{z-2} \right|^2 & = \sum_{k=1}^n X_k^{z-2} \times \sum_{k'=1}^n X_{k'}^{z^*-2} \qquad \qquad \text{ ($\cdot^*$ denotes complex conjugation) }\\
& = \sum_k \sum_{k'} (X_kX_{k'})^{c-2} \left(\frac{X_k}{X_{k'}}\right)^{i\omega} \\
& = \sum_k \sum_{k'} (X_kX_{k'})^{c-2} \cos\left(\omega \log \frac{X_k}{X_{k'}} \right).
\end{align*} 
This suggests to approximate (\ref{eqn:intbias2}) by
\[\widehat{I}_c(T) \doteq \frac{1}{2\pi n^2}\sum_k \sum_{k'} (X_kX_{k'})^{c-2}  \int_{-T}^T \left((c(c-1)-\omega^2)^2+(2c-1)^2 \omega^2\right) \cos\left(\omega \log \frac{X_k}{X_{k'}} \right) \,d\omega,\]
for some value $T$. Note that the antiderivative of $\left((c(c-1)-\omega^2)^2+(2c-1)^2 \omega^2\right) \cos\left(\omega \log \frac{X_k}{X_{k'}} \right)$ is available in closed form, which makes evaluating the integral very easy. The only stumbling stone is that the integral actually diverges for $T \to \infty$, which is not surprising here given that it would essentially reflect the integrated squared `second derivative' of $\P_n = n^{-1} \sum_{k} \delta_{X_k}$. It is, therefore, paramount to select an appropriate value of $T$.

\ppn A thoughtful choice can be made by noting from the definition (\ref{eqn:Mellindef}) that, for a fixed $\Re(z) = c$, $\Ms(f;z)$ is symmetric around $\Im(z) = 0$, i.e.\ the real axis, and always reaches its maximum at $\Im(z) = 0$. In addition, $\Ms(f;z)$ typically tends to 0 quickly as one moves away from the real axis. In particular, $\Gamma(z)$ is known to be $O(e^{-\frac{1}{2}\pi |z|})$ as $\Im(z) \to \infty$ \citep[Lemma 3.2]{Paris01}, and most of the common densities on $\R^+$, being essentially Meijer densities as observed in Section \ref{subsec:Meijerdens}, show similar Gamma factors in their Mellin transform. It turns out that $\Ms(\P_n;z)$ is remarkably accurate at reconstructing $\Ms(f;z)$ over a substantial set of values of $\Im(z)$ around the real axis, that is, where it matters. The approximation badly deteriorates as $\Im(z)$ grows, but there we know that $|\Ms(f;z)| \simeq 0$ anyway. This is illustrated in Figures \ref{fig:MexLN} and \ref{fig:MexExp} for the case of the log-Normal distribution and the Exponential distribution.
 
\begin{figure}[h]
\centering
\includegraphics[width=0.7\textwidth]{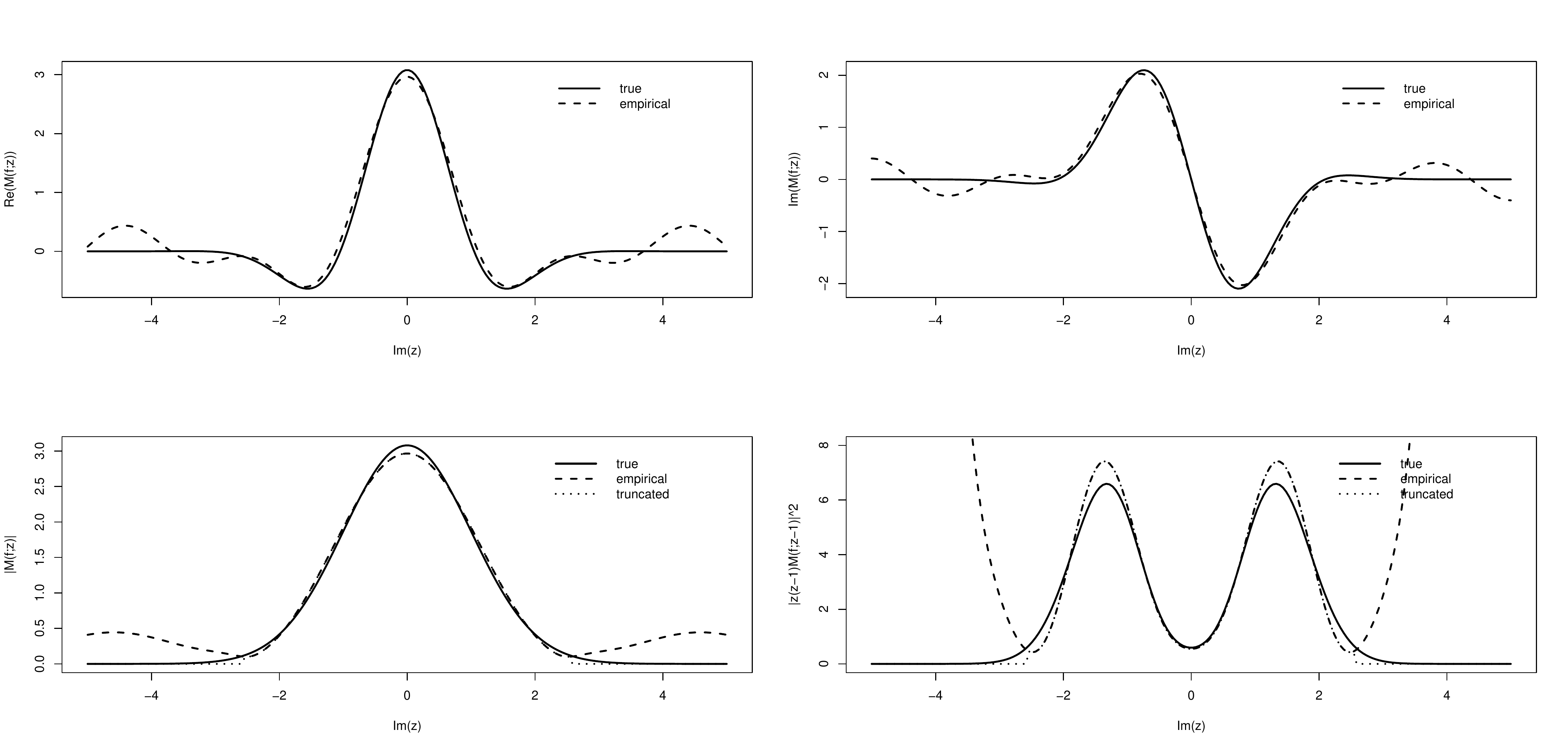}
\caption{$\Re(\Ms(f;z))$, $\Im(\Ms(f;z))$, $|\Ms(f;z)|$ and $|z(z-1)\Ms(f;z-1)|^2$ along the vertical line $\Re(z) = 1/2$ for $f$ the standard log-Normal density (plain line). Dashed lines show the empirical approximations $\Re(\Ms(\P_n;z))$, $\Im(\Ms(\P_n;z))$, $|\Ms(\P_n;z)|$ and $|z(z-1)\Ms(\P_n;z-1)|^2$ from a typical sample of size $n=500$. Dotted lines show the truncated versions of the last 2 when $\Ms(\P_n;z)$ is set to 0 outside $(-T_0,T_0)$, where $T_0$ is the location of the first local minimum of $|\Ms(\P_n;z)|$ away from $\Im(z) =0$.}
\label{fig:MexLN}
\end{figure}

\begin{figure}[h]
\centering
\includegraphics[width=0.7\textwidth]{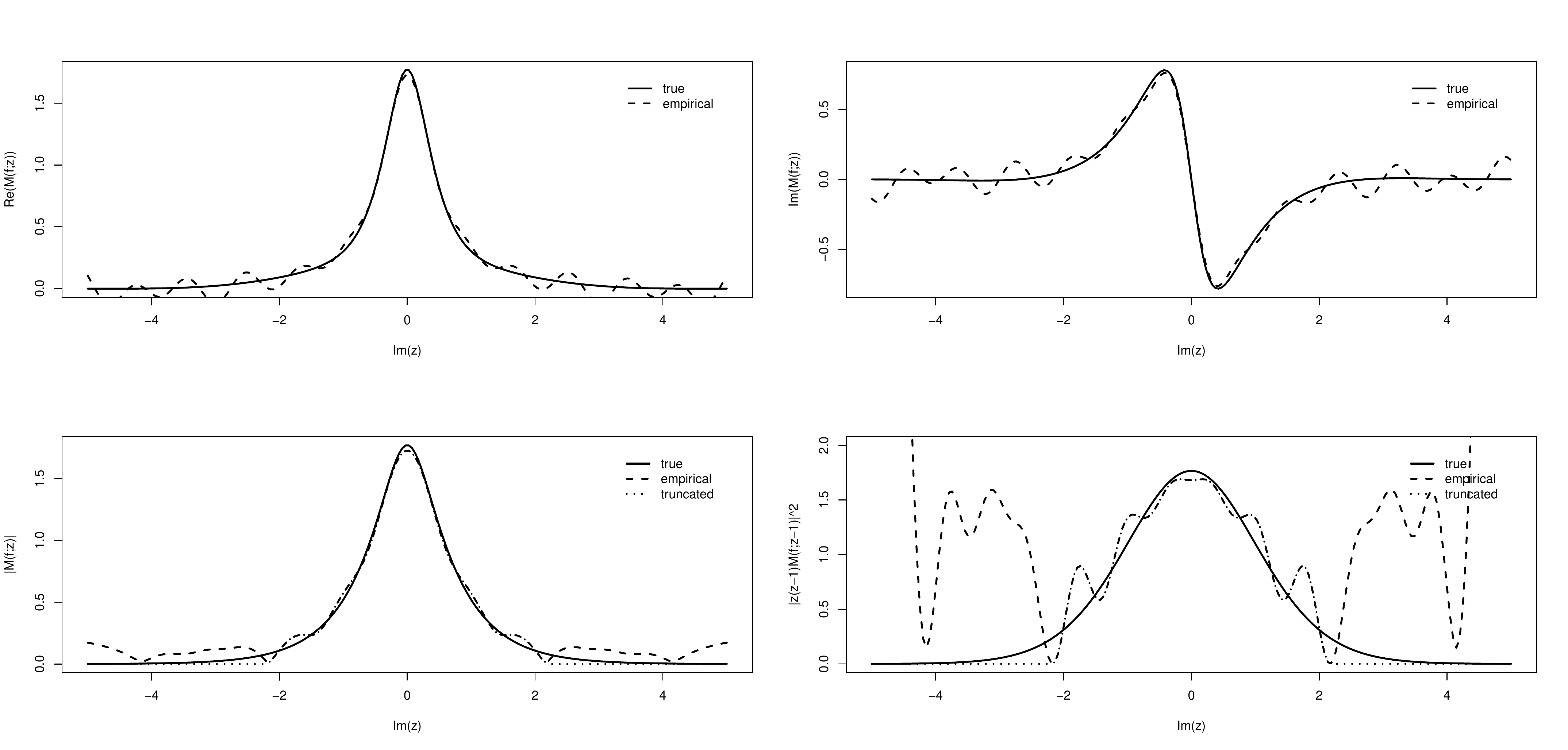}
\caption{$\Re(\Ms(f;z))$, $\Im(\Ms(f;z))$, $|\Ms(f;z)|$ and $|z(z-1)\Ms(f;z-1)|^2$ along the vertical line $\Re(z) = 3/2$ for $f$ the standard Exponential density (plain line). Dashed lines show the empirical approximations $\Re(\Ms(\P_n;z))$, $\Im(\Ms(\P_n;z))$, $|\Ms(\P_n;z)|$ and $|z(z-1)\Ms(\P_n;z-1)|^2$ from a typical sample of size $n=500$. Dotted lines show the truncated versions of the last 2 when $\Ms(\P_n;z)$ is set to 0 outside $(-T_0,T_0)$, where $T_0$ is the location of the first local minimum of $|\Ms(\P_n;z)|$ away from $\Im(z) =0$.}
\label{fig:MexExp}
\end{figure}

\ppn A reasonable choice for $T$ seems, therefore, the location, say $T_0$, of the first local minimum in $|\Ms(\P_n;z)|$ away from $\Im(z) =0$. Typically, this is where the `empirical' oscillations which heavily affect further approximations, start. It also occurs where the true $|\Ms(f;z)|$ is already `small', hence neglecting any contribution to (\ref{eqn:intbias2}) made for $\Im(z) > T_0$ is unlikely to change the outcome dramatically; see Figures \ref{fig:MexLN} and \ref{fig:MexExp}. This suggests to take, finally,
\begin{equation} \eta = \left(\frac{(2\sqrt{\pi})^{-1} \frac{1}{n} \sum_{k=1}^n X_k^{2c-3/2}}{\widehat{I}_c(T_0)} \right)^{1/5} n^{-1/5}. \label{eqn:etaopthat} \end{equation}
The value of $c$ should, of course, be chosen in (\ref{eqn:ccond}), to guarantee that all involved quantities are finite. Taking $c=3/2$ seems to be a reasonable default value, as it always belongs to (\ref{eqn:ccond}) under the mild moment conditions $\alpha \geq 1/2$ and $\beta \geq 3/2$ (and $\xi \geq 1/2$). If the behaviour of $f$ at 0 allows it, one can also take $c=1/2$, which would approximate the usual MISE-optimal bandwidth. In any case, (\ref{eqn:etaopthat}) is shown to consistently produce reliable estimates in the simulation study in the next section.

\section{Simulation study} \label{sec:sim}

In this section the practical performance of the Mellin-Meijer kernel density estimator (\ref{eqn:MKDEK}) is analysed through simulations. Inspired by \cite{Bouezmarni05}, we consider the following 10 test densities (refer to Table \ref{tab:meijerdistr} for parameterisation), shown in Figure \ref{fig:densities}:
\begin{enumerate}[(1)]\itemsep0em 
	\item the standard log-Normal density;
	\item the Chi-squared density with $k=1$ degree of freedom;
	\item the Nakagami density with $m = 1$ and $\Omega=2$;
	\item the Gamma density with $\alpha = 2$ and $\beta = 1/2$;
	\item the Gamma density with $\alpha = 0.7$ and $\beta = 1/2$;
	\item the standard Exponential density;
	\item the Generalised Pareto density with $\sigma=2/3$ and $\zeta = 2/3$;
	\item the inverse Weibull density with $\mu=1$ and $\eta = 2$; 
	\item a mixture of Gamma densities: $2/3 \times \Gamma(0.7,1/2) + 1/3 \times \Gamma(20,5)$;
	\item a mixture of log-Normal densities: $2/3 \times \text{log-}\Ns(0,1) + 1/3 \times \text{log-}\Ns(1.5,0.1)$.
\end{enumerate}

\begin{figure}[h]
	\centering
	\includegraphics[width=1\textwidth]{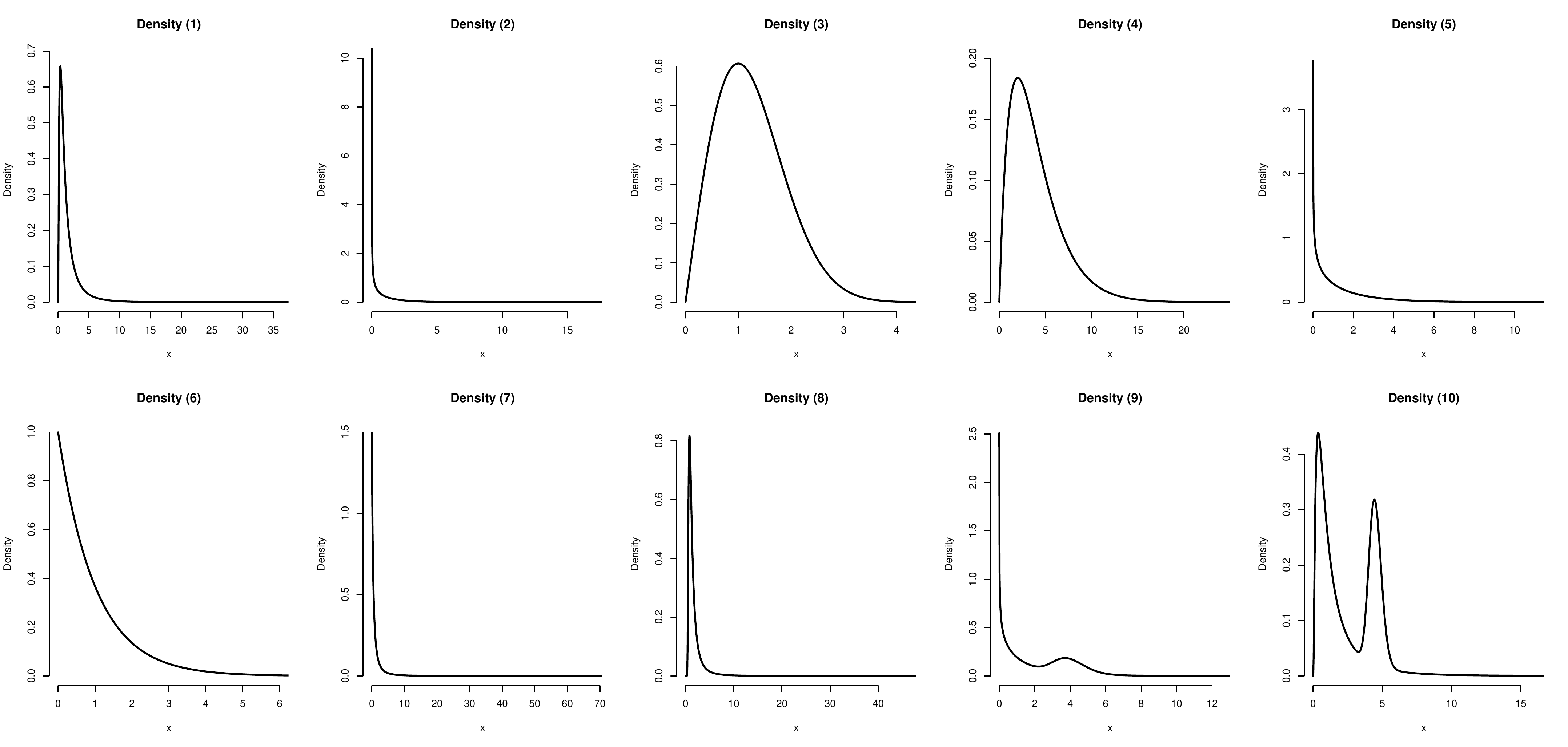}
	\caption{Densities used in the simulation study.}
	\label{fig:densities}
\end{figure}

\ppn These 10 densities exhibit various behaviours at 0 (light head: (1), (8), (10); fat head ($\E(X^{-3/2})<\infty$): (3), (4); {\it very} fat head ($\E(X^{-3/2})=\infty$): (6), (7) (bounded), (2), (5), (9) unbounded) and in the tail (light tail: (1), (2), (3), (4), (5), (6), (9), (10); fat tail: (7), (8)). From each of these distributions, independent samples of size $n=100$ and $n=500$ were generated, with $M=1,000$ Monte Carlo replications for each sample size. On each of them, the density was estimated by the estimator (\ref{eqn:MKDEK}), where the basic parameters $(\xi,\theta)$ of the Meijer kernel were set to $(1,\pi/4)$, $(1/2,0)$ and $(2,\pi/2)$ (`MM-1', `MM-2' and `MM-3' in Table \ref{tab:100}). For each case, the smoothing parameter $\eta$ was selected according to (\ref{eqn:etaopthat}), where three values $c$ were tested: $c=1/2$, $c=1$ and $=3/2$ (MM-x-12, MM-x-22 and MM-x-32)  in Table \ref{tab:100}).

\ppn For comparison, we have also included in the study \cite{Chen00}'s `modified' Gamma kernel estimator (`Gamma' in Table \ref{tab:100}), whose bandwidth was chosen following the reference rule prescribed in \citet{Hirukawa14}. This estimator was computed using the {\tt dbckden} function in the R package {\tt evmix}.

\ppn The densities were estimated on a fine grid of $N=1,000$ points between $q_{0.9999}/1000$ and $q_{0.9999}$, where $q_{0.9999}$ is the quantile of level 0.9999 of the relevant density. The Mean Integrated Squared Error (MISE) of a given estimator $\hat{f}$ was then approximated by
\begin{align*} \widehat{\text{MISE}}\left(\hat{f}\right) & =  \frac{1}{M} \sum_{q=1}^{M} \frac{1}{N} \sum_{i=1}^{N} \left(\hat{f}\left(\frac{i \times q_{0.9999}}{N}\right)-f\left(\frac{i\times q_{0.9999}}{N}\right)\right)^2, \end{align*}
where $M=1,000$ is the number of Monte Carlo replications. The results are reported in Table \ref{tab:100} for $n=100$. The results for $n=500$ show a very similar pattern and are omitted here. For ease of reading and interpretation, all the values in Table \ref{tab:100} are relative to the (approximated) MISE of the Gamma kernel estimator, which is taken as benchmark owing to its central and reference role within the asymmetric kernel density estimators. For reference, its effective MISE ($ \times 10^4$) is reported in italics in the second row of the table (which are, therefore, {\it not} on the same scale as the other values).

\begin{table}[H]
	
	\small
\begin{tabular}{||c | c c c c c c c c c c || }
	\hline \hline
	& Dens 1 & Dens 2 & Dens 3 & Dens 4 & Dens 5 & Dens 6 & Dens 7 & Dens 8 & Dens 9 & Dens 10  \\
	\hline 
Gamma&\uline{1.0000}&1.0000&1.0000&1.0000&\uline{1.0000}&\bf 1.0000&\bf 1.0000&1.0000&\bf 1.0000&1.0000 \\
 & \it 2.52 & \it 15.04 & \it 27.89 & \it 1.59 & \it 5.37 & \it 9.39 & \it 0.02& \it  3.08 & \it 6.70 & \it 8.04 \\
\hdashline
%Bound. &1.2121&1.5478&0.7775&0.9235&1.2293&1.2698&2.8150&0.5891&0.9883&0.6283 \\
%\hdashline
MM-1-12&1.1423&60.5327&\uline{0.8151}&0.8668&16.3165&7.1636&7.1307&0.5167&10.5386&\bf 0.7328 \\
MM-1-22&0.9929&6.0001&0.9142&0.8704&3.0499&1.7499&1.8415&\bf 0.4838&1.7679&1.0435 \\
MM-1-32&1.0267&{\bf 0.4933}&0.8263&{\bf 0.8310}&\uline{0.9904}&1.4479&1.1517&0.5052&\uline{1.0581}&1.5039 \\
MM-2-12&1.1365&60.5346&0.8712&\uline{0.8507}&16.2919&7.0522&7.1313&0.5144&10.5138&\uline{0.7334} \\
MM-2-22&1.0608&5.9855&1.0394&0.9920&2.9871&1.6653&1.8440&\uline{0.4851}&1.7170&1.1043 \\
MM-2-32&1.1962&0.5274&{\bf 0.8116}&0.8611&1.1069&1.5037&\uline{1.0181}&0.5304&1.1575&1.6735 \\
MM-3-12&1.1459&60.5320&\uline{0.8235}&0.8812&16.3201&7.1824&7.1304&0.5178&10.5421&0.7335 \\
MM-3-22&{\bf 0.9865}&6.0018&0.9144&0.8738&3.0635&1.7788&1.8405&0.4848&1.7790&1.0352 \\
MM-3-32&\uline{1.0057}&0.4975&0.8402&0.8466&{\bf 0.9898}&1.4613&1.1706&0.5034&\uline{1.0502}&1.4776 \\
\hline \hline
\end{tabular}
\caption{(approximated) MISE relative to the MISE of the modified Gamma kernel estimator (second row), $n=100$. Bold values show the minimum MISE for the corresponding density (non-significantly different values are underlined).}
\label{tab:100}
\end{table}

\ppn Table \ref{tab:100} confirms the aptitude of Mellin-Meijer kernel estimation. Specifically, there is a MM-estimator which outperforms (Densities 2, 3, 4, 8, 10), sometimes by a large extent (half MISE for Densities 2 and 8), or is on par with (Densities 1, 5, 7, 9) the Gamma kernel estimator. A notable exception is Density 6 (Exponential), for which the Gamma estimator does better. The `modified' Gamma kernel estimator is actually so designed for staying bounded at $x=0$ \citep[p.\,473]{Chen00}. So it is especially good at estimating densities such as the Exponential. This may sometimes be counterproductive, though, see next Section. The MM-kde is not doing bad either, in any case (for $c=3/2$, its MISE is less than 1.5 the MISE of the Gamma estimator). 

\ppn The results also show that the choice of the parameters $\xi$ and $\theta$ has little influence on the MISE of the estimator, although it may have a more important effect on a particular estimate as individual inspection reveals. As a guideline, it seems better to use $\xi=1/2$ and/or $\theta=0$ when $f$ is suspected to be positive or unbounded at $x=0$, and $\theta=\pi/2$ when $f$ is expected to have a fat tail. If both, the `default' choice $\pi/4$ works fine with $\xi$ `small', for instance $\xi=1/2$. 

\ppn The parameter which does have a great impact on the final estimate is, as usual, the smoothing parameter $\eta$. The results show that the selector (\ref{eqn:etaopthat}) is good at picking a right value of $\eta$ if used with an appropriate value of $c$. In particular, we get huge MISE's for Densities 2, 5, 6, 7 and 9 if (\ref{eqn:etaopthat}) is computed with $c=1/2$. Those are the densities such that $\E(X^{-3/2}) = \infty$. As (\ref{eqn:intbias2}) involves $\Ms(f;z-1) = \E(X^{z-2})$ along the line $\Re(z)=c$, the empirical $\widehat{I}_c(T)$ is huge for $c=1/2$, which obviously produces a heavily undersmoothed bandwidth $\eta$. For those densities, the selector is doing very good with $c=3/2$. For the other densities, the value of $c$ is less important. Usually $c=3/2$ works well in most situations.

\section{Real data analyses} \label{sec:realdat}

This section illustrates the performance of the Mellin-Meijer kernel estimator (\ref{eqn:MKDEK}) when estimating two $\R^+$-supported densities from real data. The first data set is the `suicide' data set,\footnote{This data set is directly available from the R package {\tt bde}, among others.} which gives the lengths (in days) of $n=86$ spells of psychiatric treatment undergone by patients used as controls in a study of suicide risks. Originally reported by \cite{Copas80}, it was studied among others in \cite{Silverman86} and \cite{Chen00} in relation to boundary issues: indeed, visual inspection (raw data at the bottom of the graph, histogram) reveals that the density should be positive, if not unbounded, at $x=0$, making the conventional estimator (\ref{eqn:convkde}) clearly unsuitable. The anticipated `fat head' suggests the choice $\theta=0$ and $\xi =1/2$ for the Meijer kernels. The smoothing parameter returned by (\ref{eqn:etaopthat}) with $c=3/2$ is $\eta = 4.74$. Figure \ref{fig:suicide} (left panel) shows the estimated density. The estimate shows a spike at the 0 boundary, which is easily understood. There are 3 observations exactly equal to 1 in the data set, and at this scale, this is pretty much `on the boundary'. Hence the estimator attempts to put a positive probability mass atom at 0, producing the spike. Away from the boundary, the estimate decays readily and smoothly.

\begin{figure}[h]
\centering
\includegraphics[width=0.8\textwidth]{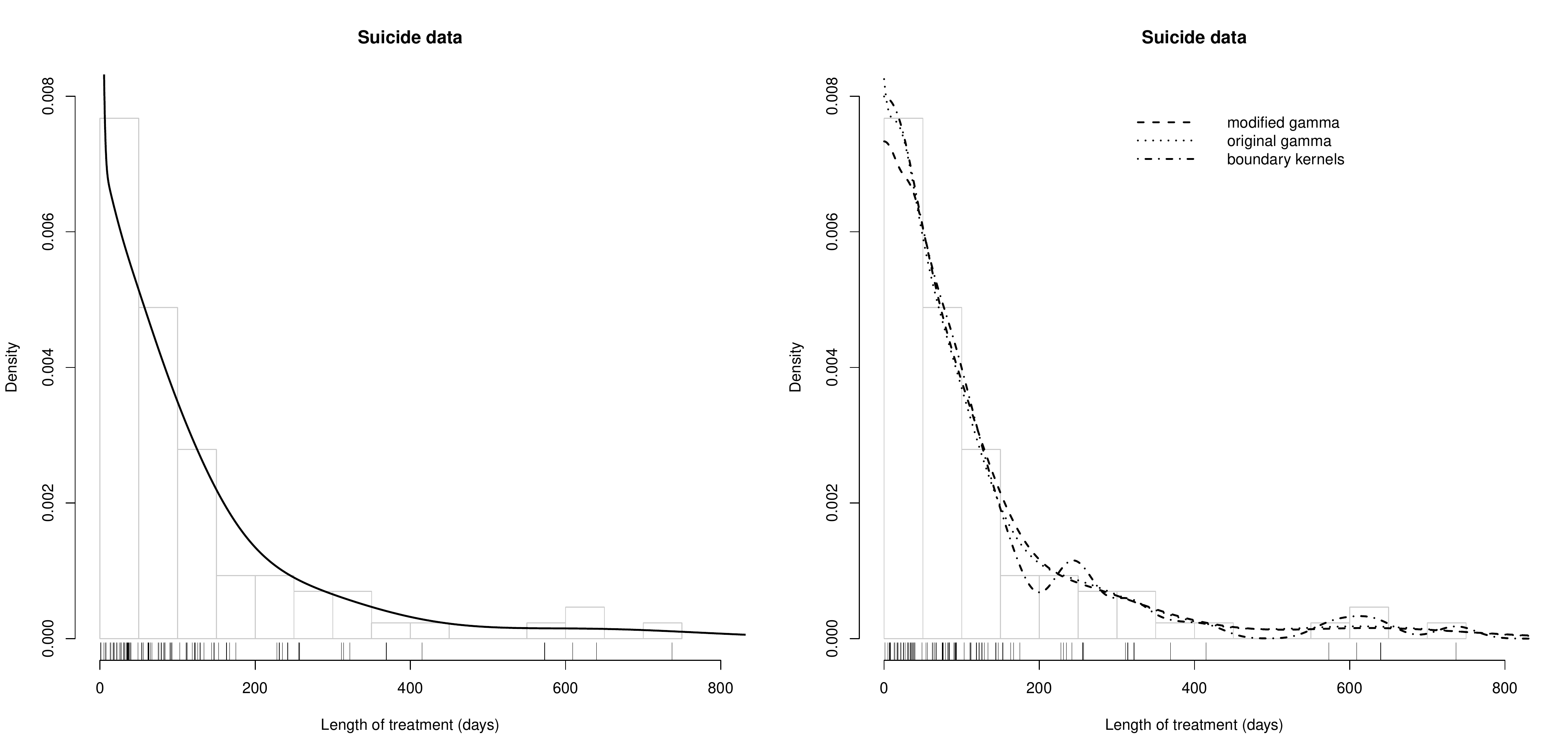}
\caption{`Suicide' data set: Mellin-Meijer kernel density estimator with $\theta = 0$, $\xi =1/2$ and $\eta =4.74$ (left panel); Two Gamma estimators and a `boundary-corrected' conventional kernel estimator (right panel).}
\label{fig:suicide}
\end{figure}

\ppn For comparison, the two Gamma kernel estimators (`original' (\ref{eqn:gammakde}) and `modified'), with bandwidths chosen by reference rule \citep{Hirukawa14}, as well as the `boundary-corrected' conventional estimator \citep{Jones96} with \cite{Sheather91}'s bandwidth, are shown in the right panel. While the Gamma kernel estimators behave very similarly to the Mellin-Meijer kernel estimator in the tail, their behaviour at the boundary is not satisfactory. The original Gamma shows an inelegant kink, whereas the modified Gamma seems to underestimate $f$ there, compared to the other estimates and the histogram. This is typical of the modified Gamma estimator, as discussed in \cite{Malec12}. The boundary-corrected kernel estimate fails to show a real peak at $x=0$, and exhibits numerous `spurious bumps' in the right tail for $x>200$. 

\ppn In the second example we estimate the World Distribution of Income. Estimating such distribution is important as various measures of growth, poverty rates, poverty counts, income inequality or welfare at the scale of the world are based on it \citep{Pinkovskiy09}. We obtained data\footnote{Data available on request.} about the GDP per capita (in constant 2000 international dollars) of $n=182$ countries in 2003 from the World Bank Database. Raw data are shown in Figure \ref{fig:WDI} along with an histogram and the estimated density by the MM-kernel estimator. We set $\theta = \pi/4$ and $\xi = 1$ (`default' choice), and the value returned by (\ref{eqn:etaopthat}) with $c=3/2$ was $\eta = 28.54$. 

\begin{figure}[h]
\centering
\includegraphics[width=0.5\textwidth]{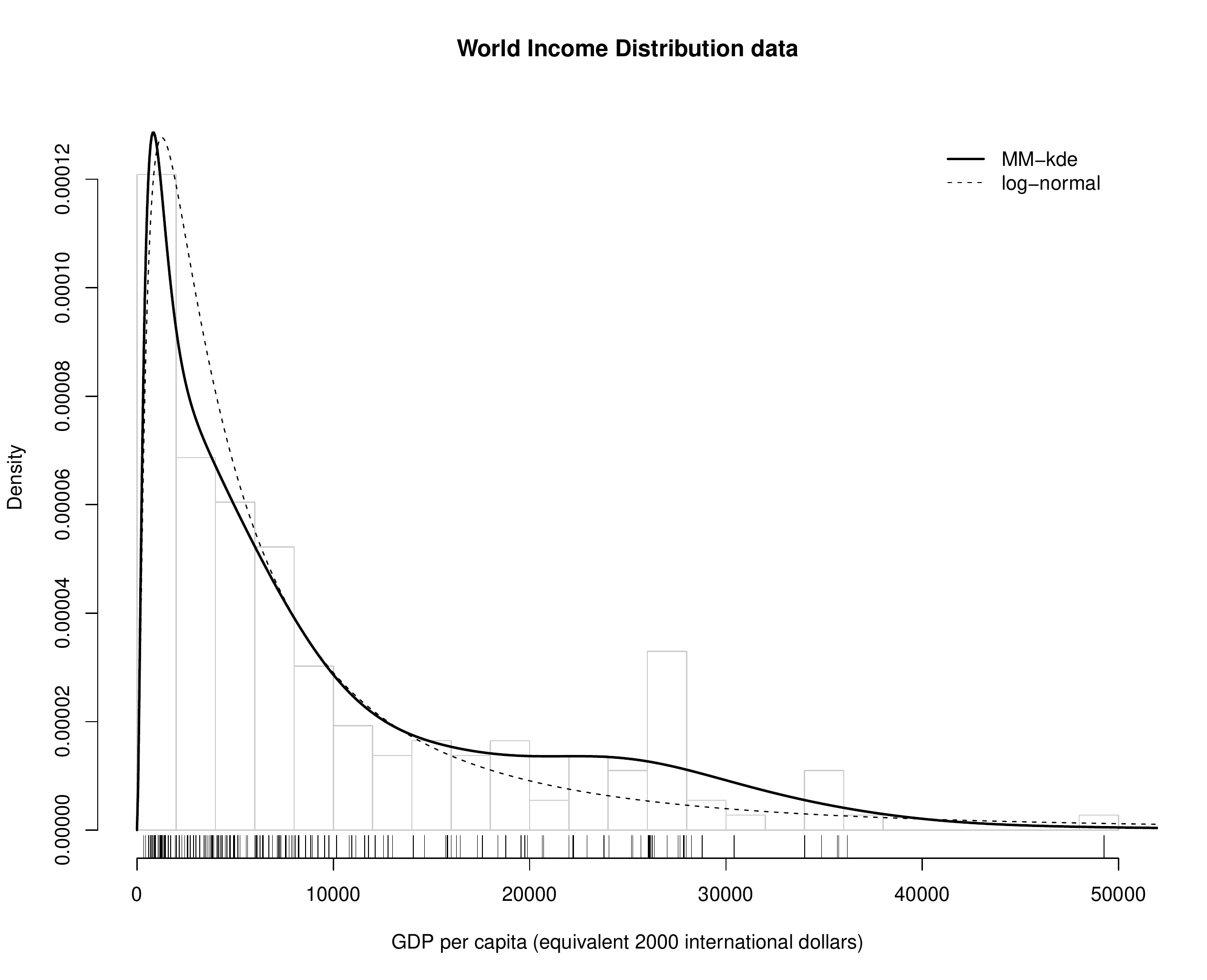}
\caption{`World Distribution of Income' data set: Mellin-Meijer kernel density estimator with $\theta = \pi/4$, $\xi =1$ and $\eta =28.54$ (plain line) and Maximum Likelihood log-Normal parametric fit (dashed line).}
\label{fig:WDI}
\end{figure}

\ppn A log-Normal parametric density, fitted by Maximum Likelihood ($\hat{\mu} = 8.58$, $\hat{\sigma} = 1.20$), is also shown in Figure \ref{fig:WDI}. \cite{Pinkovskiy09} strongly advocated in favour of the log-Normal distribution for modelling these data. However, the nonparametric, (mostly) unconstrained MM-estimate reveals that the peak close to 0 is actually narrower than the `log-Normal peak', whereas there are much more countries with GDP per capita in the range 15,000 - 40,000 than what the log-Normal distribution prescribes. In other words, analysis through the log-Normal model is likely to underestimate poverty and income inequality at the world level. See further discussion in \citet[Section 4]{Dai10}.

\section{Concluding remarks and perspectives} \label{sec:ccl}

Within his seminal works on compositional data, i.e., data living on the simplex, \citet[Section 1.8.1]{Aitchison05} already noted: ``{\it For every sample space there are basic group operations which, when recognized,
dominate clear thinking about data analysis.}'' He continued: ``{\it In $\R^d$, the two operations, translation and scalar multiplication, are so familiar that their fundamental role is often overlooked}'', implying that, when not in $\R^d$ (in his case: in the simplex), there is no reason to blindly stick to those operations. The methodology developed in this paper perfectly aligns with this stance. It has apparently been largely overlooked in earlier literature that the `boundary issues' of the conventional kernel density estimator find their very origin in that $\R^+$ equipped with the addition $+$ is not a group. Noting that the natural group operation on $\R^+$ is the multiplication $\times$, we have investigated a new kind of kernel estimation for $\R^+$-supported probability densities which achieves smoothing through `multiplicative dilution', as opposed to `additive dilution' for the conventional kernel estimator. 

\ppn The construction gives rise to an estimator which makes use of asymmetric kernels, although in a different way to most of other estimators known under that name, such as the Gamma kernel estimator \citep{Chen00}. Unlike those competitors, our estimator is based on a valid smoothing operation on $\R^+$, namely the Mellin convolution, which avoids any inconsistency in the definition and the behaviour of the estimator. Owing to the strong connection between the Mellin convolution and Meijer's $G$-functions, we have proposed to use so-called Meijer densities as kernels. Meijer distributions form a huge class of distributions supported on $\R^+$ which includes most of the classical distributions of interest, and have tractable Mellin transforms given as a product of Gamma functions. This produces an integrated theory for such `{\it Mellin-Meijer kernel density estimation}', with general features no more specific to a particular choice of kernel. The numerous pleasant properties of the estimator have been expounded in the paper.

\ppn The idea can be extended to more general settings in a straightforward way. Suppose that the density of a random variable living on a given domain $D$ is to be estimated from a sample. If $D$ can be equipped with an operation $\bullet$ making  $(D,\bullet)$ a group, then a natural kernel estimator can easily be constructed by diluting any observation through $\bullet$-convoluting it with a random disturbance in $D$. For instance, \citet[Section 2.8]{Aitchison86} defined the perturbation operator $\oplus$ as the fundamental group operation on the simplex. Thus, proper kernel density estimation on the simplex (and that includes univariate density estimation on $[0,1]$) should be performed by `$\oplus$-dilution' of each observation. This will be investigated in more details in a follow-up paper.

\ppn Interestingly, \citet[Section 2.4.2]{Aitchison05} already introduced the Mellin transform as the suitable analytical tool for simplicial distributions. More generally, the Mellin transform of a probability density $f$ ought to be a fundamental function in statistics. In some sense, it is more natural than the characteristic function (the Fourier transform of $f$) itself, as it just explicitly returns the moments of $f$ (real, complex, integral and fractional) -- \cite{Nair39} initially called it the ``{\it moment function}''. It is, therefore, rather surprising that Mellin-inspired procedures have stayed this inconspicuous in the statistical literature so far. Historically, one can find papers investigating statistical applications of the Mellin transform only intermittently over decades \citep{Epstein48,Levy59,Dolan64,Springer66,Springer70,Lomnicki67,Subrahmaniam70} or \cite{Gray88}. Only recently has the Mellin transform made a (discreet) resurgence in the statistical literature, e.g.\ in \cite{Tagliani01}, \cite{Nicolas02}, \cite{Cottone10}, \cite{Balakrishnan14} or \cite{Belomestny15,Belomestny16}. Those papers testify of the appropriateness of the Mellin transform and Mellin convolution in any multiplicative framework, such as problems of multiplicative censoring for instance. We hope that the present paper will humbly contribute to that resurgence.

\appendix 

\section*{Appendix} \label{sec:app}

\section{Further properties of Mellin transforms} \label{app:A}

\subsection{Operational properties} \label{app:A1}

Further properties of the Mellin transform include:
\begin{align}
\Ms\left(\sum_{q=1}^Q a_q f_q;z\right) & = \sum_{q=1}^Q a_q \Ms(f_q;z) & z \in \bigcap_{q=1}^Q \Ss_{f_q}, a_1,\ldots, a_Q \in \R^+ \label{eqn:A1} \\
 \Ms(f(ax);z) & = a^{-z} \Ms(f;z)  & z \in \Ss_f, a \in \R^+  \label{eqn:A2} \\
 \Ms(f(x^a);z) & = \frac{1}{|a|} \Ms(f;\frac{z}{a}) & \frac{z}{a} \in \Ss_f, a \in \R_0  \label{eqn:A3} \\
\Ms(x^y f(x);z) & = \Ms(f;z+y) & z+y \in \Ss_f, y \in \C  \label{eqn:A4} \\
\Ms(\log^n(x) f(x);z) & = \frac{d^n\Ms(f;z)}{dz^n} & z \in \Ss_f, n \in \N  \label{eqn:A5}\\
\Ms\left(\frac{d^n f}{dx^n}(x);z\right) & = (-1)^n (z-n)(z-n+1)\ldots(z-1) \Ms(f;z-n) & z-n \in \Ss_f, n \in \N \label{eqn:A6}\\
\Ms\left(x^n\frac{d^n f}{dx^n}(x);z\right) & = (-1)^n z(z+1)\ldots(z+n-1) \Ms(f;z) & z \in \Ss_f, n \in \N \label{eqn:A7} \\
\Ms\left(\left( x\frac{d }{dx}\right)^n f(x);z\right) & = (-1)^n z^n \Ms(f;z) & z \in \Ss_f, n \in \N, \label{eqn:A8}
\end{align}
where $\left( x\frac{d }{dx}\right)$ is the differential operator. Proofs of these results can be found in the above-mentioned references. In addition, the Mellin version of Parseval's identity \citep[Equation (3.1.15)]{Paris01} reads
\begin{equation} \int_0^\infty x^{2c-1} f^2(x) \,dx = \frac{1}{2\pi} \int_{\Re(z) = c} |\Ms(f;z)|^2\,dz, \label{eqn:Parseval} \end{equation}
for any $c \in \Ss_f$. 

\ppn The following useful result about Mellin transforms of probability densities easily follows from these properties.

\begin{lemma} \label{prop:scalpow} Let $X$ be a continuous positive random variable with density $f_X$ whose Mellin transform is $\Ms(f_X;z)$ on the strip of holomorphy $\Ss_{f_X} = \{ z\in \C: a<\Re(z) < b \}$, for some $a<1<b$. Then, the random variable $Y = \nu X^\xi$, where $\nu \in \R_0^+$ and $\xi \in \R$, has a density $f_Y$ whose Mellin transform is $\Ms(f_Y;z) = \nu^{z-1} \Ms\left(f_X;1+ \xi(z-1)\right)$ on $\Ss_{f_Y} = \{z \in \C: 1- \frac{1-a}{\xi} < \Re(z) < 1+\frac{b-1}{\xi}\}$ ($\xi>0$) or $\Ss_{f_Y} = \{z \in \C: 1-\frac{b-1}{|\xi|} < \Re(z) < 1+ \frac{1-a}{|\xi|} \}$ ($\xi <0$).
\end{lemma}
\begin{proof} Standard developments show that 
\begin{equation} f_Y(y) = \frac{1}{\nu |\xi|} \left(\frac{y}{\nu}\right)^{\frac{1}{\xi}-1} f_X\left(\left(\frac{y}{\nu}\right)^\frac{1}{\xi}\right), \qquad y > 0. \label{eqn:powerdens} \end{equation}
Then (\ref{eqn:A1}), (\ref{eqn:A2}), (\ref{eqn:A3}) and (\ref{eqn:A4}) directly yield the result. \end{proof}

\begin{corollary} \label{cor:inv} Let $X$ be a continuous positive random variable with density $f_X$ whose Mellin transform is $\Ms(f_X;z)$ on the strip of holomorphy $\Ss_{f_X} = \{ z\in \C: a<\Re(z) < b \}$, for some $a<1<b$. Then the inverse random variable $Y = 1/X$ has density $f_Y$ whose Mellin transform is $\Ms(f_Y;z) = \Ms(f_X;2-z)$ on $\Ss_{f_Y} = \{z \in \C: 2-b < \Re(z) < 2-a \}$.
\end{corollary}
\begin{proof} Take $\nu=1$ and $\xi = -1$ in Lemma \ref{prop:scalpow}. \end{proof}

\subsection{Meijer parameterisation} \label{app:meijerdistr}

\begin{table}[H] \centering
 \begin{tabular}{|l l c|c c c c| }
\hline
Common name & Density & Parameters & $\nu$ & $\gamma$ & $\xi$ & $\theta$ \\
\hline
Beta prime & $f(x) = \frac{x^{\alpha-1}(1+x)^{-\alpha-\beta}}{\Bs(\alpha,\beta)}$  & $\alpha,\beta>0$ & $\beta/\alpha$ & $\sqrt{\frac{1}{\alpha}+\frac{1}{\beta}}$ &  $1$ & $\tan^{-1} \sqrt{\frac{\alpha}{\beta}}$ \\
Burr & $f(x) = ck\,\frac{x^{c-1}}{(1+x^c)^{k+1}}$  & $c,k>0$ & $1$ & $ \frac{1}{c} \sqrt{1+\frac{1}{k}}$ & $\frac{1}{c}$ &  $\tan^{-1} \sqrt{\frac{1}{k}}$ \\ 
Chi & $f(x) = \frac{2^{1-k/2}x^{k-1}e^{-x^2/2}}{\Gamma\left(\frac{k}{2}\right)}$  & $k>0$ & $\sqrt{k}$ & $\frac{1}{\sqrt{2k}} $  & $\frac{1}{2}$ & $0$ \\ 
Chi-squared & $f(x) = \frac{x^{\frac{k}{2}-1}e^{-x/2}}{2^{k/2}\Gamma\left(\frac{k}{2}\right)}$  & $k>0$ & $k$ & $\sqrt{\frac{2}{k}} $ & $1$ & $0$ \\
Dagum & $f(x) = \frac{apx^{ap-1}}{b^{ap}\left(1+\left(\frac{x}{b} \right)^a \right)^{p+1}}$ & $a,b,p>0$ & $b$ & $\frac{1}{a} \sqrt{1+\frac{1}{p}}$ & $\frac{1}{a}$ & $\tan^{-1}\sqrt{\frac{1}{p}}$ \\  
Erlang & $f(x) = \frac{x^{k-1}e^{-x/\mu}}{\mu^k(k-1)!}$ & $\mu>0, k \in \N$ & $\mu k$ & $\sqrt{\frac{1}{k}}$ & 1 & $0$ \\  
Fisher-Snedecor & $f(x) = \frac{\left(d_1/d_2\right)^{d_1/2}}{\Bs\left(\frac{d_1}{2},\frac{d_2}{2}\right)} \frac{x^{\frac{d_1}{2}-1}}{\left(1+\frac{d_1}{d_2}\, x \right)^{\frac{d_1+d_2}{2}}}$ & $d_1,d_2 >0$  & 1 & $\sqrt{\frac{2}{d_1}+\frac{2}{d_2}}$ & 1 & $\tan^{-1}\sqrt{\frac{d_1}{d_2}}$ \\
Fr\'echet & $f(x) = \frac{\alpha}{s} \left(\frac{x}{s} \right)^{-1-\alpha} e^{-\left(\frac{x}{s} \right)^{-\alpha}}$ & $\alpha,s>0$ & $s$ & $\frac{1}{\alpha}$ & $\frac{1}{\alpha}$ & $\frac{\pi}{2}$ \\  
Gamma & $f(x) = \frac{\beta^\alpha}{\Gamma(\alpha)}x^{\alpha-1} e^{-\beta x}$ & $\alpha,\beta > 0$ & $\frac{\alpha}{\beta}$ & $\sqrt{\frac{1}{\alpha}}$ & 1 & 0 \\
Generalised Pareto & $f(x) = \frac{1}{\sigma} \left(1+ \frac{\zeta x}{\sigma} \right)^{-\frac{1}{\zeta}-1}$ & $\sigma,\zeta >0$ & $\frac{\sigma}{\zeta}$ & $\sqrt{\zeta+1}$ & $1$ & $\tan^{-1}\sqrt{\zeta}$ \\
Inverse Gamma & $f(x) = \frac{\beta^\alpha}{\Gamma(\alpha)}x^{-\alpha-1} e^{-\beta/x}$ & $\alpha,\beta >0$ & $\frac{\beta}{\alpha}$  & $\sqrt{\frac{1}{\alpha}}$ &1 & $\frac{\pi}{2}$ \\ 
L\'evy & $f(x) = \sqrt{\frac{c}{2\pi}} \frac{e^{-c/(2x)}}{x^{3/2}}$ & $c > 0$ & $c$ & $\sqrt{2}$ & 1 & $\frac{\pi}{2}$ \\
Log-logistic & $f(x) = \frac{\beta}{\alpha}\,\frac{\left(\frac{x}{\alpha} \right)^{\beta-1}}{\left(1+\left(\frac{x}{\alpha} \right)^\beta \right)^2}$ & $\alpha,\beta > 0$ & $\alpha$ & $\frac{\sqrt{2}}{\beta}$ & $\frac{1}{\beta}$ & $\frac{\pi}{4}$ \\
Maxwell & $f(x) = \frac{\sqrt{2}}{\sigma^3 \sqrt{\pi}}\,x^2 e^{-x^2/(2\sigma^2)}$ & $\sigma > 0$ & $\sqrt{3}\,\sigma$ & $\frac{1}{\sqrt{6}}$ & $ \frac{1}{2}$ & 0 \\
Nakagami & $f(x) = \frac{2m^m}{\Gamma(m) \Omega^m}\,x^{2m-1}e^{-\frac{mx^2}{\Omega}}$ & $m, \Omega > 0$ & $\sqrt{\Omega}$ & $\frac{1}{2\sqrt{m}}$ & $\frac{1}{2}$ & 0 \\
Rayleigh & $f(x) = \frac{x}{\sigma^2} e^{-\frac{x^2}{2\sigma^2}}$ & $\sigma >0$ & $\sqrt{2}\,\sigma$ & $\frac{1}{2}$ & $\frac{1}{2}$ & 0 \\
Singh-Maddala & $f(x) = \frac{aq}{b}\,\frac{x^{a-1}}{\left(1+\left(\frac{x}{b} \right)^a \right)^{q+1}}$ & $a,b,q >0$ & $\frac{b}{q^{1/a}}$ & $\frac{1}{a}\sqrt{1+\frac{1}{q}}$ & $\frac{1}{a}$ & $\tan^{-1}\sqrt{\frac{1}{q}}$ \\
Stacy & $f(x) = \frac{1}{\Gamma\left(\frac{d}{p}\right)} \frac{p}{a^d} x^{d-1} e^{-(x/a)^p}$ & $a,d,p>0$ & $a\left(\frac{d}{p}\right)^{1/p}$ & $\frac{1}{p}\sqrt{\frac{d}{p}}$ & $\frac{1}{p}$ & 0 \\
Weibull & $f(x) = \frac{\eta}{\mu}\left(\frac{x}{\mu} \right)^{\eta-1} e^{-\left(\frac{x}{\mu} \right)^\eta}$ & $\mu,\eta>0$ & $\mu$ & $\frac{1}{\eta}$ & $\frac{1}{\eta}$ & 0 \\
\hline

 \end{tabular}
\caption{Meijer parameterisation ($\nu$, $\gamma$, $\xi$, $\theta$ in (\ref{eqn:G11})) for the most common $\R^+$-supported densities.} \label{tab:meijerdistr}
\end{table}

\section{Proofs}

\subsection*{Preliminary lemma}

First we state a technical lemma that will be used repeatedly in the proofs below. \cite{Tricomi51} gave the following asymptotic expansion for the ratio of two Gamma functions. Let $t, \alpha, \beta \in \C$. Then, as $|t| \to \infty$,
\begin{equation} \frac{\Gamma(t + \alpha)}{\Gamma(t+\beta)} = t^{\alpha-\beta}\left(\sum_{k=0}^{M-1} \frac{1}{k!} \frac{\Gamma(1+\alpha-\beta)}{\Gamma(1+\alpha -\beta -k)} B_{k}^{(1+\alpha-\beta)}(\alpha) \frac{1}{t^{k}} + R_{M}(t)\right), \label{eqn:tricomerdel} \end{equation}
where $|R_M(t)| = O(|t|^{-M})$, provided that $|\alpha|, |\beta|$ are bounded and $|\arg(t+\alpha)|<\pi$. Here $B_k^{(a)}(x)$ are the generalised Bernoulli polynomials, which are polynomials in $a$ and $x$ of degree $k$, see \citet[Section 1.1]{Temme96}. The first such polynomials are $B_0^{(a)}(x) = 1$, $B_1^{(a)}(x) = x -a/2$ and $B_2^{(a)}(x) = (3a^2 +12x^2 - a(1+12x))/12$. The following result, proved in \cite{Fields70}, essentially gives a uniform version of (\ref{eqn:tricomerdel}) which allows $|\alpha|$ and $|\beta|$ to become `large' as well, if more slowly than $|t|$.  
\begin{lemma} \label{lemma:fields} Let $t, \alpha, \beta \in \C$. Then, for all $M = 1, 2,\ldots$, one has, as $|t| \to \infty$, 
\begin{equation} \frac{\Gamma(t + \alpha)}{\Gamma(t+\beta)} = t^{\alpha-\beta}\left(\sum_{k=0}^{M-1} \frac{1}{k!} \frac{\Gamma(1+\alpha-\beta)}{\Gamma(1+\alpha -\beta -k)} B_{k}^{(1+\alpha-\beta)}(\alpha) \frac{1}{t^{k}} + R_{M}(t,\alpha,\beta)\right), \label{eqn:fields} \end{equation}
where $|R_M(t,\alpha,\beta)| =O\left(|t|^{-M} (1+|\alpha-\beta|)^M(1+|\alpha|+|\alpha-\beta|)^M))\right)$, provided that $|\arg (t+\alpha)| < \pi$ %$\beta -\alpha \notin \N$ 
and $(1+|\alpha-\beta|)(1+|\alpha|+|\alpha-\beta|)=o(|t|)$.
\end{lemma}

\subsection*{Proof of Proposition \ref{prop:Meijerkern}}

For any two parameters $\alpha_1>0$ and $\alpha_2>0$, the Fisher-Snedecor distribution $F(2\alpha_1,2\alpha_2)$ has density 
\begin{equation} f_\text{F}(x) = \frac{1}{\Bs(\alpha_1,\alpha_2)}\left(\frac{\alpha_1}{\alpha_2} \right)^{\alpha_1} \frac{x^{\alpha_1-1}}{\left(1+\frac{\alpha_1}{\alpha_2}x\right)^{\alpha_1+\alpha_2}}, \qquad x>0,  \label{eqn:Fdens} \end{equation}
where $\Bs(\cdot,\cdot)$ is the Beta function. One of its characterisations is that, if $X_1 \sim \text{Gamma}(\alpha_1,\beta_1)$ and $X_2 \sim \text{Gamma}(\alpha_2,\beta_2)$ for some arbitrary positive $\beta_1$ and $\beta_2$ ($X_1$ and $X_2$ independent), then
\[\frac{\alpha_2 \beta_1}{\alpha_1 \beta_2} \frac{X_1}{X_2}  \sim F(2\alpha_1,2\alpha_2),\]
see \citet[Section 27.8]{Johnson94}. Hence the $F$-distribution is the distribution of the product of a Gamma$(\alpha_1,\beta_1)$ r.v.\ and an Inverse Gamma$(\alpha_2,\beta_2)$ r.v., rescaled by the constant $\alpha_2 \beta_1/\alpha_1 \beta_2$. Lemma \ref{prop:scalpow}, (\ref{eqn:MellProd}), (\ref{eqn:gamMT}) and (\ref{eqn:invgamMT}) then yield the Mellin transform of (\ref{eqn:Fdens}):
\begin{align}
 \Ms(f_\text{F};z) & = \left(\frac{\alpha_2 \beta_1}{\alpha_1 \beta_2}\right)^{z-1} \times \frac{1}{\beta_1^{z-1}}\frac{\Gamma(\alpha_1 + z-1)}{\Gamma(\alpha_1)} \times \frac{1}{\beta_2^{1-z}}\frac{\Gamma(\alpha_2 + 1-z)}{\Gamma(\alpha_2)} \notag \\
& = \left(\frac{\alpha_2 }{\alpha_1}\right)^{z-1} \frac{\Gamma(\alpha_1 + z-1)\Gamma(\alpha_2 + 1-z)}{\Gamma(\alpha_1)\Gamma(\alpha_2)}, \qquad 1- \alpha_1 < \Re(z) < 1+ \alpha_2. \label{eqn:FMT}
\end{align}

\ppn Identifying (\ref{eqn:MTkern}) and (\ref{eqn:FMT}), $L_{1,\gamma,1,\theta}$ is seen to be the $F\left(\frac{2}{\gamma^2 \cos^2 \theta},\frac{2}{\gamma^2 \sin^2 \theta}\right)$-density. From Lemma \ref{prop:scalpow}, it follows that $L_{\nu,\gamma,\xi,\theta}$ is the density of the positive random variable $Y = \nu X^\xi$, where $X \sim F\left(\frac{2\xi^2}{\gamma^2 \cos^2 \theta},\frac{2\xi^2}{\gamma^2 \sin^2 \theta}\right)$.

\subsection*{Proof of Proposition \ref{prop:momkern}}

If $\xi > 2\gamma^2 \sin^2 \theta$, then $\{z \in \C: 2\leq \Re(z) \leq 3\} \subset \Ss_{L_{\nu,\gamma,\xi,\theta}}$. Then, by (\ref{eqn:Mellmoments}), the mean $\mu$ of $L_{\nu,\gamma,\xi,\theta}$ is  $\mu = \int_0^\infty x L_{\nu,\gamma,\xi,\theta}(x)\,dx = \Ms(L_{\nu,\gamma,\xi,\theta};2)$, that is, 
\begin{equation} \mu = \nu  \left(\frac{1}{\tan^2 \theta} \right)^{\xi}\  \frac{\Gamma\left(\frac{\xi^2}{\gamma^2 \cos^2 \theta}+\xi \right) \Gamma\left(\frac{\xi^2}{\gamma^2 \sin^2 \theta}-\xi \right)}{\Gamma\left(\frac{\xi^2}{\gamma^2 \cos^2 \theta}\right)\Gamma\left(\frac{\xi^2}{\gamma^2 \sin^2 \theta}\right)}. \label{eqn:muetak}\end{equation}
Also, as $\int_0^\infty x^2 L_{\nu,\gamma,\xi,\theta}(x)\,dx = \Ms(L_{\nu,\gamma,\xi,\theta};3)$, the standard deviation of $L_{\nu,\gamma,\xi,\theta}$ is $$\sigma = \sqrt{\Ms(L_{\nu,\gamma,\xi,\theta};3)-\Ms^2(L_{\nu,\gamma,\xi,\theta};2)},$$
which is 
\begin{multline*} \sigma  = \nu  \left(\frac{1}{\tan^2 \theta} \right)^{\xi}\  \frac{\Gamma\left(\frac{\xi^2}{\gamma^2 \cos^2 \theta}+\xi \right) \Gamma\left(\frac{\xi^2}{\gamma^2 \sin^2 \theta}-\xi \right)}{\Gamma\left(\frac{\xi^2}{\gamma^2 \cos^2 \theta}\right)\Gamma\left(\frac{\xi^2}{\gamma^2 \sin^2 \theta}\right)} \\ \times \sqrt{\frac{\Gamma\left(\frac{\xi^2}{\gamma^2 \cos^2 \theta}\right)\Gamma\left(\frac{\xi^2}{\gamma^2 \sin^2 \theta}\right)\Gamma\left(\frac{\xi^2}{\gamma^2 \cos^2 \theta}+2\xi\right)\Gamma\left(\frac{\xi^2}{\gamma^2 \sin^2 \theta}-2\xi\right)}{\Gamma^2\left(\frac{\xi^2}{\gamma^2 \cos^2 \theta}+\xi\right)\Gamma^2\left(\frac{\xi^2}{\gamma^2 \sin^2 \theta}-\xi\right)}-1}.\end{multline*}
The announced result follows from $\chi = \frac{\sigma}{\mu}$.

\subsection*{Proof of Proposition \ref{lem:asympCV}}

The proof is given for the case $\theta \notin \{0,\pi/2\}$ only. By (\ref{eqn:tricomerdel}), we have, as $\gamma \to 0$,
\begin{multline*} \frac{\Gamma\left(\frac{\xi^2}{\gamma^2 \cos^2 \theta}\right)}{\Gamma\left(\frac{\xi^2}{\gamma^2 \cos^2 \theta}+\xi\right)}\times \frac{\Gamma\left(\frac{\xi^2}{\gamma^2 \cos^2 \theta}+2\xi\right)}{\Gamma\left(\frac{\xi^2}{\gamma^2 \cos^2 \theta}+\xi\right)} = \left(\frac{\xi^2}{\gamma^2 \cos^2 \theta}\right)^{-\xi}\left( 1+\frac{ 1-\xi}{2\xi}\,\gamma^2 \cos^2 \theta+ O(\gamma^4)\right) \\ \times \left(\frac{\xi^2}{\gamma^2 \cos^2 \theta}\right)^{\xi}\left( 1+\frac{3\xi-1}{2\xi}\,\gamma^2 \cos^2 \theta+ O(\gamma^4)\right) 
\end{multline*}
\[ = 1 + \gamma^2 \cos^2 \theta + O(\gamma^4).\]
Similarly, 
\[ \frac{\Gamma\left(\frac{\xi^2}{\gamma^2 \sin^2 \theta}\right)}{\Gamma\left(\frac{\xi^2}{\gamma^2 \sin^2 \theta}-\xi\right)}\times \frac{\Gamma\left(\frac{\xi^2}{\gamma^2 \sin^2 \theta}-2\xi\right)}{\Gamma\left(\frac{\xi^2}{\gamma^2 \sin^2 \theta}-\xi\right)} = 1 + \gamma^2 \sin^2 \theta + O(\gamma^4).\]
Hence 
\[\frac{\Gamma\left(\frac{\xi^2}{\gamma^2 \cos^2 \theta}\right)\Gamma\left(\frac{\xi^2}{\gamma^2 \cos^2 \theta}+2\xi\right)}{\Gamma^2\left(\frac{\xi^2}{\gamma^2 \cos^2 \theta}+\xi\right)}\, \frac{\Gamma\left(\frac{\xi^2}{\gamma^2 \sin^2 \theta}\right)\Gamma\left(\frac{\xi^2}{\gamma^2 \sin^2 \theta}-2\xi\right)}{\Gamma^2\left(\frac{\xi^2}{\gamma^2 \sin^2 \theta}-\xi\right)} = 1+ \gamma^2 +O(\gamma^4), \]
and the result follows from (\ref{eqn:fullgamma}). %In the case $\xi \in \N$, then the ratios in (\ref{eqn:fullgamma}) can be worked out explicitly using the recursion $\Gamma(\alpha+1) = \alpha \Gamma(\alpha)$, and the result easily follows as well. 

%\frac{\Gamma\left(\frac{\xi^2}{\gamma^2 \cos^2 \theta}\right)\Gamma\left(\frac{\xi^2}{\gamma^2 \sin^2 \theta}\right)\Gamma\left(\frac{\xi^2}{\gamma^2 \cos^2 \theta}+2\xi\right)\Gamma\left(\frac{\xi^2}{\gamma^2 \sin^2 \theta}-2\xi\right)}{\Gamma^2\left(\frac{\xi^2}{\gamma^2 \cos^2 \theta}+\xi\right)\Gamma\left(\frac{\xi^2}{\gamma^2 \sin^2 \theta}-\xi\right)}

\subsection*{Proof of Proposition \ref{prop:asympMell}}

The proof is given for the case $\theta \notin \{0,\pi/2\}$ only. As $\gamma \to 0$, Lemma \ref{lemma:fields} ascertains that
\begin{align*} \left(\frac{\xi^2}{\gamma^2 \cos^2 \theta}\right)^{-\xi(z-1)} \frac{\Gamma\left(\frac{\xi^2}{\gamma^2 \cos^2 \theta}+\xi(z-1)\right)}{\Gamma\left(\frac{\xi^2}{\gamma^2 \cos^2 \theta}\right)} & = 1+ \frac{\gamma^2 \cos^2 \theta}{2\xi^2}\xi(z-1)(\xi(z-1)-1) +\rho_{1}(\gamma,z) \\ 
& = 1+ \frac{\gamma^2 \cos^2 \theta}{2} (z-1)(z-1-\frac{1}{\xi})+\rho_{1}(\gamma,z),\end{align*}
where $|\rho_{1}(\gamma,z)| =O(\gamma^4 (1+|z-1|)^2)$, provided $|z-1|=o(\gamma^{-2})$. Similarly, 
\begin{align*} \left(\frac{\xi^2}{\gamma^2 \sin^2 \theta}\right)^{-\xi(1-z)}\frac{\Gamma\left(\frac{\xi^2}{\gamma^2 \sin^2 \theta}+\xi(1-z)\right)}{\Gamma\left(\frac{\xi^2}{\gamma^2 \sin^2 \theta}\right)} & = 1+ \frac{\gamma^2 \sin^2 \theta}{2\xi^2} \xi(1-z)(\xi(1-z)-1) +\rho_{2}(\gamma,z) \\ 
& = 1+ \frac{\gamma^2 \sin^2 \theta}{2}(z-1)(z-1+\frac{1}{\xi}) +\rho_{2}(\gamma,z),\end{align*}
where $|\rho_{2}(\gamma,z)| =O(\gamma^4 (1+|z-1|)^2)$, provided $|z-1|=o(\gamma^{-2})$. Also, the binomial series expands as 
\[\nu^{z-1} = (1+\Delta \gamma^2)^{z-1} = 1+\Delta \gamma^2 (z-1)+\rho_{3}(\gamma,z), \]
where $|\rho_{3}(\gamma,z)| =O(\gamma^4 |(z-1)(z-2)|)$, provided $|z-1| =o(\gamma^{-2})$ as $\gamma \to 0$.
Multiplying these factors yields
\begin{align*} \nu^{z-1}\ \left(\frac{1}{\tan^2 \theta}\right)^{\xi(z-1)}\ &\  \frac{\Gamma\left(\frac{\xi^2}{\gamma^2 \cos^2 \theta}+\xi(z-1)\right)}{\Gamma\left(\frac{\xi^2}{\gamma^2 \cos^2 \theta}\right)} \  \frac{\Gamma\left(\frac{\xi^2}{\gamma^2 \sin^2 \theta}+\xi(1-z)\right)}{\Gamma\left(\frac{\xi^2}{\gamma^2 \sin^2 \theta}\right)} \\ & =  1+ \frac{\gamma^2}{2}\,(z-1) \left(\cos^2 \theta\  (z-1-\frac{1}{\xi})+\sin^2 \theta\  (z-1+\frac{1}{\xi}) +2\Delta\right)+\rho(\gamma,z) \\
& = 1+ \frac{\gamma^2}{2}\,(z-1) \left(z-1 -\frac{\cos 2\theta}{\xi}+2\Delta\right)+\rho(\gamma,z) \end{align*}
where $|\rho(\gamma,z)|=O(\gamma^4(1+|z-1|)^2)$, provided $|z-1|=o(\gamma^{-2})$. %If $\theta \in \{0,\pi/2\}$, then the results follows as well by excluding the irrelevant factors as per the lines following Proposition \ref{prop:momkern}. 

\subsection*{Proof of Proposition \ref{prop:asympMellL2}}

$(i)$ From \citet[item (15), p.\ 349]{Bateman54}, it can be seen that 
\[\left\{G_{1,1}^{1,1}\left(\cdot \left| \substack{1-b \\ a}\right.  \right)\right\}^2= \frac{\Gamma^2(a+b)}{\Gamma(2a+2b)}G_{1,1}^{1,1}\left(\cdot \left| \substack{1-2b \\ 2a}\right.  \right), \]
which yields
\[\Ms\left(\left\{G_{1,1}^{1,1}\left(\cdot \left| \substack{1-b \\ a}\right.  \right)\right\}^2 ;z\right) = \frac{\Gamma^2(a+b)}{\Gamma(2a+2b)} \Gamma(z+2a)\Gamma(2b-z),\]
by (\ref{eqn:MTG11}). Combining this with (\ref{eqn:G11}) leads to (\ref{eqn:MTL2})-(\ref{eqn:SL2}) after some algebraic work. 

\ppn $(ii)$ Resorting to Lemma \ref{lemma:fields}, one obtains, as $\gamma \to 0$,
\[\nu^{z-2}\left(\frac{1}{\tan^2 \theta}\right)^{\xi(z-2)}\,\frac{\Gamma\left(\frac{2\xi^2}{\gamma^2 \cos^2 \theta}+\xi(z-2) \right)}{\Gamma\left(\frac{2\xi^2}{\gamma^2 \cos^2 \theta}\right)} \, \frac{\Gamma\left(\frac{2\xi^2}{\gamma^2 \sin^2 \theta}+\xi(2-z) \right)}{\Gamma\left(\frac{2\xi^2}{\gamma^2 \sin^2 \theta}\right)} = 1+\omega(\gamma,z), \]
where $|\omega(\gamma,z)| = O(\gamma^2 (1+|z-2|))$ for $|z-2| = o(\gamma^{-2})$. On the other hand, for any $a,b >0$, 
\begin{align*}
 \frac{\Bs\left(2a,2b \right)}{\Bs^2\left(a,b \right)} & = \frac{\Gamma(2a)}{\Gamma(a)}\frac{\Gamma(2b)}{\Gamma(b)}\frac{\Gamma(a+b)}{\Gamma(2(a+b))}\frac{\Gamma(a+b)}{\Gamma(a)\Gamma(b)} \\
& = \frac{1}{2\sqrt{\pi}} \frac{\Gamma(a+1/2)}{\Gamma(a)}\frac{\Gamma(b+1/2)}{\Gamma(b)} \frac{\Gamma(a+b)}{\Gamma(a+b+1/2)} \qquad \text{(duplication formula)}.
\end{align*}
Now, as $a, b \to \infty$, use (\ref{eqn:tricomerdel}) and see
 \begin{align*} \frac{\Bs\left(2a,2b \right)}{\Bs^2\left(a,b \right)} &=\frac{1}{2\sqrt{\pi}}\, a^{1/2} (1+O(a^{-1}))\,b^{1/2} (1+O(b^{-1}))\,(a+b)^{-1/2} (1+O(a+b)^{-1})) \\ 
& = \frac{1}{2\sqrt{\pi}} \frac{1}{\left(\frac{1}{a}+\frac{1}{b}\right)^{1/2}}\,(1+O(a^{-1})+O(b^{-1})+O((a+b)^{-1})).\end{align*}
With $a = \frac{\xi^2}{\gamma^2 \cos^2 \theta}$ and $b = \frac{\xi^2}{\gamma^2 \sin^2 \theta}$, see that $1/a + 1/b = \gamma^2/\xi^2$, hence
\[ \frac{\Bs\left(\frac{2\xi^2}{\gamma^2 \cos^2 \theta},\frac{2\xi^2}{\gamma^2 \sin^2 \theta} \right)}{\Bs^2\left(\frac{\xi^2}{\gamma^2 \cos^2 \theta},\frac{\xi^2}{\gamma^2 \sin^2 \theta} \right)}  = \frac{1}{2\sqrt{\pi}}\,\frac{\xi}{\gamma} (1+O(\gamma^2)).\]
It follows
\[\Ms(L^2_{\nu,\gamma,\xi,\theta};z) = \frac{1}{2\sqrt{\pi}}\,\frac{1}{\gamma} \,(1+\omega(\gamma,z)),\]
where $|\omega(\gamma,z)| = O(\gamma^2 (1+|z-2|))$ for $|z-2| = o(\gamma^{-2})$.

\subsection*{Proof of Theorem \ref{thm:consistMISE}}

Apply Parseval's identity (\ref{eqn:Parseval}) to $\hat{f} - f$ to get 
\begin{equation} \int_0^\infty x^{2c-1} \left(\hat{f}(x) - f(x)\right)^2  \,dx = \frac{1}{2\pi} \int_{\Re(z) = c} |\Ms(\hat{f}-f;z)|^2\,dz, \label{eqn:ISE} \end{equation}
for any $c\in \Ss_{\hat{f} - f}$. Then we resort to the following lemma.

\begin{lemma} \label{lem:Sfhat} Under Assumptions \ref{ass:iid}-\ref{ass:eta}, the strip of holomorphy $\Ss_{\hat{f} - f}$ of $\hat{f} - f$ is such that
\begin{equation} \left\{z \in \C: 1-\min(\alpha,\xi/\cos^2\theta) \leq \Re(z) \leq 1+\min(\beta,\xi/\sin^2\theta) \right\} \subseteq \Ss_{\hat{f} - f}. \label{eqn:SfSfhat} \end{equation}
\end{lemma}
\begin{proof} From (\ref{eqn:MKDEconvK}) and (\ref{eqn:A1}),
\begin{equation} \Ms(\hat{f};z)  = \frac{1}{n} \sum_{k=1}^n \Ms(L_\eta^{(k)};z) X_k^{z-1}, \label{eqn:MTfhat} \end{equation}
with $\Ss_{\hat{f}}= \bigcap_{k=1}^n \Ss_{L_\eta^{(k)}}$. From (\ref{eqn:SL}) and (\ref{eqn:coefgamma}), we see that
\begin{align*} \Ss_{L_\eta^{(k)}} & = \left\{z \in \C: 1-\frac{\xi (\eta^2 +X_k)}{\eta^2 \cos^2 \theta} < \Re(z) < 1+\frac{\xi (\eta^2 +X_k)}{\eta^2 \sin^2 \theta} \right\} \\ & \supseteq  \left\{z \in \C: 1-\frac{\xi}{\cos^2 \theta} \leq \Re(z) \leq 1+\frac{\xi}{\sin^2 \theta} \right\} \quad \text{ for all } k, \end{align*}
whence 
\[\Ss_{\hat{f}} \supseteq \left\{z \in \C: 1-\frac{\xi}{\cos^2 \theta} \leq \Re(z) \leq 1+\frac{\xi}{\sin^2 \theta} \right\}. \]
Assumption \ref{ass:Sf} implies, through (\ref{eqn:Sfmom}), that 
\begin{equation} \left\{z \in \C: 1-\alpha \leq \Re(z) \leq 1+\beta \right\} \subseteq  S_f. \label{eqn:Sf} \end{equation}
The result follows as $\Ss_{\hat{f} - f} = \Ss_{\hat{f}} \cap \Ss_f$, from (\ref{eqn:A1}). 
\end{proof}

\ppn Lemma \ref{lem:Sfhat} ascertains that (\ref{eqn:ISE}) is valid for any $c \in \left[1-\min(\alpha,\xi/\cos^2\theta), 1+\min(\beta,\xi/\sin^2\theta)\right]$. In particular, it is true for $c$ satisfying (\ref{eqn:ccond}), as $1-\min(\alpha,\xi/\cos^2\theta) \leq \max(2-\alpha,1-\xi/\cos^2\theta)$ and $1+\min(\beta,\xi/\sin^2\theta) \geq \min((3+2\beta)/4,1+\xi/\sin^2\theta)$.%  because \tbc as $\alpha + \beta >1$ by Assumption \ref{ass:Sf}, $2-\alpha < 1+\beta$ and (\ref{eqn:MISE}) is valid for all $c \in (2-\alpha,1+\beta)$. 

\ppn Now, because $\Ms(\hat{f}-f;z)$ is holomorphic on $\Ss_{\hat{f}-f}$ and $\hat{f} - f$ is real-valued, $\Ms^*(\hat{f}-f;z) = \Ms(\hat{f}-f;z^*)$, where $\cdot^*$ denotes complex conjugation. Hence, $|\Ms(\hat{f}-f;z)|^2 = \Ms(\hat{f}-f;z) \times \Ms(\hat{f}-f;z^*)$. By (\ref{eqn:A1}), $\Ms(\hat{f}-f;z) = \Ms(\hat{f};z)-\Ms(f;z)$.  Hence (\ref{eqn:ISE}) is
\begin{align*} \int_0^\infty x^{2c-1} \left(\hat{f}(x) - f(x)\right)^2  \,dx  & = \frac{1}{2\pi} \int_{\Re(z) = c} \Ms(\hat{f};z)\Ms(\hat{f};z^*)\,dz \\
& -\frac{1}{2\pi} \int_{\Re(z) = c} \Ms(f;z)\Ms(\hat{f};z^*)\,dz \\
& -\frac{1}{2\pi} \int_{\Re(z) = c} \Ms(\hat{f};z)\Ms(f;z^*)\,dz \\
& +\frac{1}{2\pi} \int_{\Re(z) = c} \Ms(f;z)\Ms(f;z^*)\,dz \\
& \doteq  \text{\encircle{A}} + \text{\encircle{B}}  + \text{\encircle{C}}  + \text{\encircle{D}},
\end{align*}
and
\begin{equation} \E\left( \int_0^\infty x^{2c-1} \left(\hat{f}(x) - f(x)\right)^2  \,dx \right) =  \E\left(\text{\encircle{A}}\right) + \E\left(\text{\encircle{B}}\right)  + \E\left(\text{\encircle{C}}\right)  + \text{\encircle{D}}. \label{eqn:MISE} \end{equation}

\ppn From (\ref{eqn:MTfhat}), we have
\[\Ms(\hat{f};z)\Ms(\hat{f};z^*) = \frac{1}{n^2} \sum_{k=1}^n \Ms(L_\eta^{(k)};z) \Ms(L_\eta^{(k)};z^*) X_k^{2\Re(z)-2}  + \frac{1}{n^2} \sum_{k=1}^n \sum_{k' \neq k}  \Ms(L_\eta^{(k)};z) \Ms(L_\eta^{(k')};z^*) X_k^{z-1} X_{k'}^{z^*-1}, \]
whence
\begin{multline} \text{\encircle{A}}  = \frac{1}{n^2} \sum_{k=1}^n X_k^{2c-2} \frac{1}{2\pi} \int_{\Re(z) = c}  |\Ms(L_\eta^{(k)};z)|^2 \,dz \\ + \frac{1}{n^2} \sum_{k=1}^n \sum_{k' \neq k} \frac{1}{2\pi} \int_{\Re(z) = c} \Ms(L_\eta^{(k)};z) \Ms(L_\eta^{(k')};z^*) X_k^{z-1} X_{k'}^{z^*-1} \,dz. \label{eqn:A}\end{multline}
Given that $c \in \bigcap_{k=1}^n \Ss_{L_\eta^{(k)}}$, it holds for all $k$
\[\frac{1}{2\pi} \int_{\Re(z) = c}  |\Ms(L_\eta^{(k)};z)|^2 \,dz = \int_0^\infty x^{2c-1} {L_\eta^{(k)}}^2(x) \,dx = \Ms({L_\eta^{(k)}}^2;2c),\]
from (\ref{eqn:Parseval}) back and forth. Hence the first term in (\ref{eqn:A}), say \encircle{A}-1, is \begin{equation} \text{\encircle{A}-1}=\frac{1}{n^2} \sum_{k=1}^n X_k^{2c-2} \Ms({L_\eta^{(k)}}^2;2c). \label{eqn:AA1} \end{equation}
Note that $c \in \bigcap_{k=1}^n \Ss_{L_\eta^{(k)}} \iff 2c \in \bigcap_{k=1}^n \Ss_{{L_\eta^{(k)}}^2}$, as seen from (\ref{eqn:SL}) and (\ref{eqn:SL2}).

\ppn The second term in (\ref{eqn:A}), say \encircle{A}-2, has expectation 
\[\E\left(\text{\encircle{A}-2} \right)  = \left(1-\frac{1}{n}\right) \frac{1}{2\pi} \int_{\Re(z) = c} \E\left(\Ms(L_\eta^{(k)};z)X_k^{z-1}\right) \E\left(\Ms(L_\eta^{(k)};z^*)  X_{k}^{z^*-1}\right) \,dz \doteq \left(1-\frac{1}{n}\right)  \E\left(\text{\encircle{A}-2-a}\right), \]
for a generic $k \in \{1,\ldots,n\}$. Interchanging expectation and integral is justified as $c$ belongs to both $\Ss_f$ and $\Ss_{L_\eta^{(k)}}$ (for all $k$), making the corresponding integrals both absolutely convergent. Likewise, 
\begin{align*} \E\left(\text{\encircle{B}}\right) & = -\frac{1}{2\pi} \int_{\Re(z) = c} \Ms(f;z)\E\left(\Ms(L_\eta^{(k)};z^*)X_k^{z^*-1}\right)\,dz \\
\text{ and } \E\left(\text{\encircle{C}}\right) & = -\frac{1}{2\pi} \int_{\Re(z) = c} \Ms(f;z^*)\E\left(\Ms(L_\eta^{(k)};z)X_k^{z-1}\right)\,dz.
 \end{align*}
It is easily seen that
\begin{equation} \E\left(\text{\encircle{A}-2-a}\right) + \E\left(\text{\encircle{B}}\right) + \E\left(\text{\encircle{C}}\right) + \text{\encircle{D}}  = \frac{1}{2\pi} \int_{\Re(z) = c} \left|\E\left(\Ms(L_\eta^{(k)};z)X_k^{z-1}\right) - \Ms(f;z)\right|^2\,dz, \label{eqn:bias2} \end{equation}
which is clearly the integrated squared bias term, say $\text{IB}^2$, in the Weighted Mean Integrated Square Error expression (\ref{eqn:MISE}). The remaining $\E\left(\text{\encircle{A}-1} \right) - \frac{1}{n} \E\left(\text{\encircle{A}-2-a}\right)$ thus forms the integrated variance, say $\text{IV}$. Below, we show that $\text{IB}^2 = O(\eta^4)$ and $\text{IV} = O((n\eta)^{-1})$ as $n \to \infty$, under our assumptions.

\ppn \uline{Integrated squared bias term:} Under condition (\ref{eqn:ccond}), $c > 2-\alpha$, hence $0<\frac{c+\alpha-2}{c+\alpha-1}<1$. Let $\epsilon \doteq \epsilon_n \to 0$ as $n \to \infty$, such that $\epsilon \sim \eta^b$ for 
\begin{equation} 0< b < \frac{c+\alpha-2}{c+\alpha-1}. \label{eqn:eps}\end{equation}
Note thas this implies $\eta/\epsilon \to 0$ as $n \to \infty$. Write
\begin{equation} \Ms(L_\eta^{(k)};z)X_k^{z-1}  = \Ms(L_\eta^{(k)};z)X_k^{z-1} \indic{X_k \geq \eta^2\left(\frac{1}{\epsilon^2}-1\right)}+ \Ms(L_\eta^{(k)};z)X_k^{z-1}\indic{X_k < \eta^2\left(\frac{1}{\epsilon^2}-1\right)},\label{eqn:MLindic} \end{equation}
where $\indic{\cdot}$ is the indicator function, equal to 1 if the condition $\{\cdot\}$ is satisfied and 0 otherwise. See that $X_k \geq \eta^2\left(\frac{1}{\epsilon^2}-1\right) \iff \frac{\eta}{\sqrt{\eta^2+X_k}} \leq \epsilon \to 0$, hence one can make use of the asymptotic expansion (\ref{eqn:MLasymp})-(\ref{eqn:MLasympsimp}) with (\ref{eqn:coefgamma})-(\ref{eqn:nucoef}) to write, as $n \to \infty$,
\[ \Ms(L_\eta^{(k)};z)X_k^{z-1} \indic{X_k \geq \eta^2\left(\frac{1}{\epsilon^2}-1\right)} =  \left(1+\frac{1}{2} \frac{\eta^2}{\eta^2 + X_k}z(z-1) + R_k(\eta,z) \right) X_k^{z-1} \indic{X_k \geq \eta^2\left(\frac{1}{\epsilon^2}-1\right)}\]
where $|R_k(\eta,z)| \leq C \frac{\eta^4}{(\eta^2 + X_k)^2} (1+|z-1|)^2$ for some constant $C$. From this and (\ref{eqn:MLindic}) we have
\begin{align*} \E\left(\Ms(L_\eta^{(k)};z)X_k^{z-1}\right) - \Ms(f;z)  = & \ \E\left(X_k^{z-1}\indic{X_k \geq \eta^2\left(\frac{1}{\epsilon^2}-1\right)} \right) - \E\left(X_k^{z-1}\right) \\ 
& \ + \frac{1}{2}\eta^2 z(z-1)  \E\left(\frac{1}{\eta^2 + X_k}X_k^{z-1}\indic{X_k \geq \eta^2\left(\frac{1}{\epsilon^2}-1\right)}\right) \\
& \ + \E\left(R_k(\eta,z) X_k^{z-1} \indic{X_k \geq \eta^2\left(\frac{1}{\epsilon^2}-1\right)}\right) \\ & \  + \E\left(\Ms(L_\eta^{(k)};z)X_k^{z-1}\indic{X_k < \eta^2\left(\frac{1}{\epsilon^2}-1\right)} \right),
\end{align*}
that is,
\begin{align*} \E\left(\Ms(L_\eta^{(k)};z)X_k^{z-1}\right) - \Ms(f;z)  = & \ \frac{1}{2}\eta^2 z(z-1)  \E\left(\frac{1}{\eta^2 + X_k}X_k^{z-1}\indic{X_k \geq \eta^2\left(\frac{1}{\epsilon^2}-1\right)}\right) \\
& \ + \E\left(\left(\Ms(L_\eta^{(k)};z)-1\right)X_k^{z-1}\indic{X_k < \eta^2\left(\frac{1}{\epsilon^2}-1\right)} \right) \\
& \ + \E\left(R_k(\eta,z) X_k^{z-1} \indic{X_k \geq \eta^2\left(\frac{1}{\epsilon^2}-1\right)}\right).
\end{align*}
Hence the integrated squared bias (\ref{eqn:bias2}) is such that
\begin{align}
 \text{IB}^2 \leq & \ \frac{1}{4}\eta^4  \frac{1}{2\pi} \int_{\Re(z) = c} \left|  z(z-1)  \E\left(\frac{1}{\eta^2 + X_k}X_k^{z-1}\indic{X_k \geq \eta^2\left(\frac{1}{\epsilon^2}-1\right)}\right)\right|^2\,dz \notag \\
& \ + \frac{1}{2\pi} \int_{\Re(z) = c} \left| \E\left(\left(\Ms(L_\eta^{(k)};z)-1\right)X_k^{z-1}\indic{X_k < \eta^2\left(\frac{1}{\epsilon^2}-1\right)} \right)\right|^2\,dz \notag \\
& \ + \frac{1}{2\pi} \int_{\Re(z) = c} \left|  \E\left(R_k(\eta,z) X_k^{z-1} \indic{X_k \geq \eta^2\left(\frac{1}{\epsilon^2}-1\right)}\right)\right|^2\,dz \notag \\
& \doteq  \ \text{\encircle{E}}+ \text{\encircle{F}} + \text{\encircle{G}}. \label{eqn:EFG}
\end{align}

\ppn As $\frac{1}{(\eta^2+X_k)} \leq \frac{1}{X_k}$, $\text{\encircle{E}} \leq  \frac{1}{4}\eta^4  \frac{1}{2\pi} \int_{\Re(z) = c} \left|  z(z-1)  \E\left(X_k^{z-2}\right)\right|^2\,dz %+ \frac{1}{4}\eta^4  \frac{1}{2\pi} \int_{\Re(z) = c} \left|  z(z-1)  \E\left(\frac{\eta^2}{\eta^2 + X_k}X_k^{z-2}\indic{X_k \geq \eta^2\left(\frac{1}{\epsilon^2}-1\right)}\right)\right|^2\,dz
$. By combining (\ref{eqn:A4}) and (\ref{eqn:A7}), it is seen that $z(z-1)\E\left(X_k^{z-2}\right) = z(z-1)\Ms(f;z-1) = \Ms(xf''(x);z)$ if $z-1 \in \Ss_f$, which is the case here by (\ref{eqn:Sf}) and because $\Re(z) =c \geq 2 - \alpha$ by (\ref{eqn:ccond}). With (\ref{eqn:Parseval}), $\text{\encircle{E}}  \leq \frac{1}{4}\eta^4 \int_0^\infty x^{2c+1} f''^2(x)\,dx$, hence 
\begin{equation} \text{\encircle{E}}=O(\eta^4). \label{eqn:E}\end{equation}

\ppn Given that $c \in \bigcap_{k=1}^n \Ss_{L_\eta^{(k)}}$, $\sup_{z\in \C: \Re(z) =c} \max_{k=1,\ldots,n}|\Ms(L_\eta^{(k)};z)|\leq C$ for some constant $C$ and 
\[ \text{\encircle{F}} \leq (1+C)^2 \frac{1}{2\pi} \int_{\Re(z) = c} \left| \E\left(X_k^{z-1}\indic{X_k < \eta^2\left(\frac{1}{\epsilon^2}-1\right)} \right)\right|^2\,dz.\]
Now, 
\[ \E\left(X_k^{z-1}\indic{X_k < \eta^2\left(\frac{1}{\epsilon^2}-1\right)} \right)  = \int_0^{\eta^2\left(\frac{1}{\epsilon^2}-1\right)} x^{z-1}f(x)\,dx = \Ms\left(f(x) \indic{x< \eta^2\left(\frac{1}{\epsilon^2}-1\right)};z\right). \]
Clearly the strip of holomorphy of $f$ is contained in that of any of its restriction on $\R^+$, so by (\ref{eqn:Parseval}) again, 
\[ \text{\encircle{F}} \leq (1+C)^2 \int_0^\infty x^{2c-1} f^2(x) \indic{x< \eta^2\left(\frac{1}{\epsilon^2}-1\right)}\,dx =(1+C)^2\int_0^{\eta^2\left(\frac{1}{\epsilon^2}-1\right)} x^{2c-1} f^2(x)\,dx.\]
By Assumpion \ref{ass:Sf}, $\E(X^{-\alpha}) < \infty$, which implies $f(x) = o(x^{\alpha-1})$ as $x \to 0$. Hence
\[ \text{\encircle{F}} = o\left(\left(\frac{\eta^2}{\epsilon^2}\right)^{2c+2\alpha-2}\right), \]
following Example 4 in \citet[Section 1.1.1]{Paris01}. With $\epsilon \sim \eta^b$ and condition (\ref{eqn:eps}), it can be checked that this is
\begin{equation} \text{\encircle{F}} = o(\eta^4). \label{eqn:F}\end{equation}

\ppn Finally,  
\begin{align*} |R_k(\eta,z)|\indic{X_k \geq \eta^2\left(\frac{1}{\epsilon^2}-1\right)} & \leq C \frac{\eta^4}{(\eta^2+X_k)^2}(1+|z-1|)^2 \indic{X_k \geq \eta^2\left(\frac{1}{\epsilon^2}-1\right)} \\ 
& = C \frac{\eta^2}{\eta^2+X_k}(1+|z-1|)^2 \frac{\eta^2}{\eta^2+X_k}\indic{X_k \geq \eta^2\left(\frac{1}{\epsilon^2}-1\right)}\\
& \leq C \frac{\eta^2}{X_k} (1+|z-1|)^2 \epsilon^2,  \end{align*}
and it follows
\[\text{\encircle{G}} \leq C \frac{1}{2\pi} \eta^4 \epsilon^4 \int_{\Re(z) = c} (1+|z-1|)^2 \left|  \E\left(X_k^{z-2} \right)\right|^2\,dz. \]
The integral may be seen to be bounded by (\ref{eqn:A8}), as $z-1 \in S_f$ for $\Re(z) = c > 2-\alpha$ under condition (\ref{eqn:ccond}), hence
\begin{equation} \text{\encircle{G}}  = O(\eta^4 \epsilon^4) = o(\eta^4). \label{eqn:G} \end{equation}
It follows from (\ref{eqn:EFG}), (\ref{eqn:E}), (\ref{eqn:F}) and (\ref{eqn:G}) that 
\[\text{IB}^2 = O(\eta^4). \]

\ppn \uline{Integrated variance term:} Consider again $\epsilon \doteq \epsilon_n \to 0$ with $\eta/\epsilon \to 0$ as $n \to \infty$. Then write (\ref{eqn:AA1}) as
\begin{align} \text{\encircle{A}-1} & = \frac{1}{n^2} \sum_{k=1}^n X_k^{2c-2} \Ms({L_\eta^{(k)}}^2;2c) \indic{X_k \geq \eta^2\left(\frac{1}{\epsilon^2}-1\right)} + \frac{1}{n^2} \sum_{k=1}^n X_k^{2c-2} \Ms({L_\eta^{(k)}}^2;2c) \indic{X_k < \eta^2\left(\frac{1}{\epsilon^2}-1\right)} \label{eqn:A1var}\\ 
& \doteq \text{\encircle{A}-1-a} + \text{\encircle{A}-1-b}. \notag\end{align}
Seeing again that $X_k \geq \eta^2\left(\frac{1}{\epsilon^2}-1\right) \iff \frac{\eta}{\sqrt{\eta^2+X_k}} \leq \epsilon \to 0$, one can write the expansion (\ref{eqn:ML2asymp}) for $\Ms({L_\eta^{(k)}}^2;2c)$ in \text{\encircle{A}-1-a}, that is, making use of (\ref{eqn:coefgamma})-(\ref{eqn:nucoef}),
\[\Ms({L_\eta^{(k)}}^2;2c) = \frac{1}{2\sqrt{\pi}} \frac{\sqrt{\eta^2+X_k}}{\eta} ( 1 + \Omega_k(\eta,2c)),\]
where $|\Omega_k(\eta,2c)| =O\left(\frac{\eta^2}{\eta^2+X_k} (1+|2c-2|)\right) =O(\epsilon^2)$. Also, $\sqrt{\eta^2 + X_k}/\sqrt{X_k} = \sqrt{1+ \eta^2/X_k} \leq 1/\sqrt{1-\epsilon^2} \leq 1+\epsilon^2$, for $n$ large enough. This means that, as $n\to \infty$,
\[\Ms({L_\eta^{(k)}}^2;2c) =  \frac{1}{2\sqrt{\pi}} \frac{\sqrt{X_k}}{\eta}(1 + \Omega'_k(\eta,2c)), \]
where $|\Omega'_k(\eta,2c))|\leq C \epsilon^2$ for some constant $C$, yielding
\[ \text{\encircle{A}-1-a} = \frac{1}{n^2 \eta} \frac{1}{2\sqrt{\pi}}\sum_{k=1}^n X_k^{2c-3/2} \indic{X_k \geq \eta^2\left(\frac{1}{\epsilon^2}-1\right)}(1+O(\epsilon^2)). \] 
Assumption \ref{ass:Sf} ensures that $f(x) = o(x^{\alpha-1})$ as $x \to 0$, whence
\[\P\left(X_k \geq \eta^2\left(\frac{1}{\epsilon^2}-1\right)\right) = 1 - \int_0^{\eta^2\left(\frac{1}{\epsilon^2}-1\right)} f(x)\,dx =1-O\left(\left(\frac{\eta^2}{\epsilon^2}\right)^\alpha \right) = 1-o(1).\]
It follows
\[\E\left(\text{\encircle{A}-1-a} \right) = \frac{1}{n \eta} \frac{1}{2\sqrt{\pi}} \Ms(f;2c-1/2) \left(1-o(1) \right) \left(1+O(\epsilon^2) \right).\]
This is $O((n\eta)^{-1})$ if $2c-1/2 \in \Ss_f$, which is the case under condition (\ref{eqn:ccond}).

\ppn Now, because $2c \in \bigcap_{k=1}^n \Ss_{{L_\eta^{(k)}}^2}$, each $|\Ms({L_\eta^{(k)}}^2;2c)|$ is finite and $\max_{1 \leq k \leq n} |\Ms({L_\eta^{(k)}}^2;2c)| \leq C$, for $C$ some constant. Hence
\begin{equation} \text{\encircle{A}-1-b} \leq  \frac{C}{n^2} \sum_{k=1}^n X_k^{2c-2} \indic{X_k < \eta^2\left(\frac{1}{\epsilon^2}-1\right)}. \label{eqn:A1b} \end{equation}
Similarly to above,
\[\E\left(X^{2c-2}\indic{X < \eta^2\left(\frac{1}{\epsilon^2}-1\right)}\right) = \int_0^{\eta^2\left(\frac{1}{\epsilon^2}-1\right)} x^{2c-2} f(x)\,dx = o\left(\left(\frac{\eta^2}{\epsilon^2}\right)^{2c-2+\alpha}\right) \qquad \text{ as } n \to \infty. \]
making use again of $f(x) = o(x^{\alpha-1})$  as $x \to 0$. Taking expectations in (\ref{eqn:A1b}) yields
\[\E\left(\text{\encircle{A}-1-b}\right)= o\left(n^{-1}\,\left(\frac{\eta^2}{\epsilon^2}\right)^{2c-2+\alpha} \right). \]
It can be checked that, for $c \geq 3/4 -\alpha/2$, $\left(\frac{\eta^2}{\epsilon^2}\right)^{2c-2+\alpha} = O(\eta^{-1})$. Hence, $\E\left(\text{\encircle{A}-1-b}\right) = o((n\eta)^{-1})$, leading to 
\begin{equation} \E(\text{\encircle{A}-1}) = O((n\eta)^{-1}). \label{eqn:EA1} \end{equation}

\ppn The dominant term in $\E\left(\text{\encircle{A}-2-a}\right)$ can be understood to be $\text{\encircle{D}}$. Yet, 
\begin{align*}
 \text{\encircle{D}} & = \frac{1}{2\pi} \int_{\Re(z) = c} \Ms(f;z)\Ms(f;z^*)\,dz \\
& = \frac{1}{2\pi} \int_{\Re(z) = c} |\Ms(f;z)|^2\,dz \\
& = \int_0^\infty x^{2c-1} f^2(x)\,dx
\end{align*}
which is bounded for any $c \in \Ss_f$. Hence $\E\left(\text{\encircle{A}-2-a}\right)/n = O(n^{-1}) = o((n\eta)^{-1})$, which shows
\[\text{IV} = O((n\eta)^{-1}). \]
\qed

\subsection*{Proof of Proposition \ref{cor:optrate2}}

We just show that (\ref{eqn:E}) holds true if $\int_0^\infty (xf''(x))^2\,dx < \infty$. From (\ref{eqn:EFG}), 
\[\text{\encircle{E}}  \leq  \frac{1}{4}\eta^4  \frac{1}{2\pi} \int_{\Re(z) = c} \left|  z(z-1)  \E\left(X_k^{z-2}\indic{X_k \geq \eta^2\left(\frac{1}{\epsilon^2}-1\right)}\right)\right|^2\,dz. \]
Now, 
\[\E\left(X_k^{z-2}\indic{X_k \geq \eta^2\left(\frac{1}{\epsilon^2}-1\right)}\right) = \int_{\eta^2\left(\frac{1}{\epsilon^2}-1\right)}^\infty x^{z-2} f(x)\,dx = \Ms\left(f(x) \indic{x \geq \eta^2\left(\frac{1}{\epsilon^2}-1\right)};z-1\right).\]
The strip of holomorphy of $f(x) \indic{x \geq \cdot }$ is $(-\infty,1+\beta)$, as $f(x) \indic{x \geq \cdot} \equiv 0 $ for $x \simeq 0$ (`flat' head). So for any $c \leq 1+\beta$, 
\begin{align*}
 \text{\encircle{E}}  \leq & \ \frac{1}{4}\eta^4  \frac{1}{2\pi} \int_{\Re(z) = c} \left|  z(z-1)  \Ms\left(f(x) \indic{x \geq \eta^2\left(\frac{1}{\epsilon^2}-1\right)};z-1\right)\right|^2\,dz \\
 & = \frac{1}{4}\eta^4 \int_0^\infty x^{2c-1} (xf''(x))^2 \indic{x \geq \eta^2\left(\frac{1}{\epsilon^2}-1\right)}\,dx \\
 & \leq \frac{1}{4}\eta^4 \int_0^\infty x^{2c-1} (xf''(x))^2 \,dx,
\end{align*}
by (\ref{eqn:Parseval}). Taking $c=1/2$ yields the result, as $\int_0^\infty (xf''(x))^2 \,dx < \infty$.\qed

\subsection*{Proof of Theorem \ref{thm:bias-var}}

The proof is very similar to the proof of Theorem \ref{thm:consistMISE}, hence only a sketch is given. Using the inverse Mellin transform expression (\ref{eqn:invMellin}), we can write
\[\left( \hat{f}(x) - f(x)\right)^2  =  \left(\frac{1}{2\pi i} \int_{\Re(z)=c} x^{-z} (\Ms(\hat{f};z) - \Ms(f;z)) \,dz \right)^2,\]
where $\Ms(\hat{f};z)$ is given by (\ref{eqn:MTfhat}) and $c$ is any value in $\left[1-\min(\alpha,\xi/\cos^2\theta), 1+\min(\beta,\xi/\sin^2\theta)\right] \subseteq \Ss_{\hat{f}-f}$, from Lemma \ref{lem:Sfhat}. Expanding the square, working out the terms and taking expectations yield, after lengthy derivations,
\begin{align*}
\text{MSE}\left( \hat{f}(x)\right) =  \E\left( \left( \hat{f}(x) - f(x)\right)^2\right)  = & \Bigg\{\frac{1}{2\pi i} \int_{\Re(z)=c} x^{-z} \left(\E\left(\Ms(L_\eta^{(k)};z) X_k^{z-1}\right) - \Ms(f;z)\right)\,dz \Bigg\}^2 \\
& + \frac{1}{n}\Bigg\{ \frac{1}{2\pi i} \int_{\Re(z)=2c-1} x^{-(z+1)} \E\left(X_k^{z-1}\Ms({L_\eta^{(k)}}^2;z+1)\right)\, dz  \\ & \quad \quad \quad  - \left(\frac{1}{2\pi i}  \int_{\Re(z)=c} x^{-z} \E\left(\Ms(L_\eta^{(k)};z) X_k^{z-1}\right)\,dz \right)^2 \Bigg\}.
\end{align*}
Clearly the first term is the square of the inverse Mellin transform of $\Ms(\E(\hat{f});z) - \Ms(f;z)$, hence is the squared pointwise bias term, say $B^2(x)$, in the usual expansion of the MSE of $\hat{f}(x)$. The term in $1/n$ is thus the pointwise variance, $V(x)$.

\ppn Acting essentially as in the proof of Theorem \ref{thm:consistMISE}, in particular making use of expansion (\ref{eqn:MLasymp})-(\ref{eqn:MLasympsimp}) again, one obtains that the dominant term asymptotically in the squared bias term is 
\begin{align*}
B^2(x) \sim \frac{1}{2} & \eta^2 \frac{1}{2\pi i} \int_{\Re(z)=c} x^{-z}  z(z-1) \Ms(f;z-1)\,dz \\
& = \frac{1}{2} \eta^2 \frac{1}{2\pi i} \int_{\Re(z)=c} x^{-z}  \Ms(x f''(x);z)\,dz,
\end{align*}
from (\ref{eqn:A4}) and (\ref{eqn:A7}), provided $z-1 \in \Ss_f$, that is, $c-1 \in \Ss_f$. It follows that 
\begin{equation} B^2(x) \sim  \frac{1}{2} \eta^2 xf''(x), \qquad \text{ as } n \to \infty. \label{eqn:psb} \end{equation}

\ppn Making use of expansion (\ref{eqn:ML2asymp}), one finds that, asymptotically,
\[\E\left(X_k^{z-1}\Ms({L_\eta^{(k)}}^2;z+1)\right) \sim \frac{1}{2\sqrt{\pi}} \frac{1}{\eta} \Ms(f;z+1/2),\]
which, plugged in the expression of $V(x)$, yields the following dominant term for the pointwise variance:
\begin{align}
V(x)  \sim &\  \frac{1}{2\sqrt{\pi}} \frac{1}{n\eta} \frac{1}{2\pi i} \int_{\Re(z)=2c-1} x^{-(z+1)} \Ms(f;z+1/2) \,dz \notag \\
& = \frac{1}{2\sqrt{\pi}} \frac{1}{n\eta}\frac{1}{\sqrt{x}} \frac{1}{2\pi i} \int_{\Re(z)=2c-1} x^{-(z+1/2)} \Ms(f;z+1/2) \,dz \notag \\
&  = \frac{1}{2\sqrt{\pi}} \frac{1}{n\eta}\frac{1}{\sqrt{x}} \frac{1}{2\pi i} \int_{\Re(z)=2c-1/2} x^{-z} \Ms(f;z) \,dz  \notag \\
& = \frac{1}{2\sqrt{\pi}} \frac{f(x)}{n \eta \sqrt{x}}, \label{eqn:pv}
\end{align}
provided $2c-1/2 \in \Ss_f$.

\ppn So, (\ref{eqn:psb}) and (\ref{eqn:pv}) are the asymptotic bias and variance of $\hat{f}(x)$, provided that there exists $c \in \left[1-\min(\alpha,\xi/\cos^2\theta), 1+\min(\beta,\xi/\sin^2\theta)\right]$ such that $1-\alpha < c-1 < 1 + \beta$ and $1-\alpha < 2c-1/2 < 1 + \beta$. Assumptions \ref{ass:Sf} and \ref{ass:Lxi} ensure there is such a $c$. \qed

\end{document}